\documentclass[11pt]{article}
\usepackage[T1]{fontenc}
\usepackage{lmodern}
\usepackage{amsmath,amsthm,amssymb,amsfonts,bbm}
\usepackage{enumerate}
\usepackage{epsfig}
\usepackage[dvipsnames,svgnames,table]{xcolor}
\usepackage{pgf,tikz}
\usetikzlibrary{calc}
\usepackage{fullpage}
\usepackage[colorlinks=true,linkcolor=Blue,urlcolor=Blue,citecolor=Red]{hyperref}
\usepackage{booktabs}
\usepackage[nameinlink]{cleveref}
\usepackage{algorithm}
\usepackage{algpseudocode}
\usepackage{caption}
\usepackage{subcaption}
\usepackage{optidef}
\usepackage[shortlabels]{enumitem}

\newcommand{\F}{\mathbb{F}}

\newcommand{\Q}{\mathbb{Q}}
\newcommand{\R}{\mathbb{R}}

\newtheorem{theorem}{Theorem}[section]
\newtheorem{proposition}[theorem]{Proposition}
\newtheorem{lemma}[theorem]{Lemma}
\newtheorem{corollary}[theorem]{Corollary}

\theoremstyle{definition}
\newtheorem{definition}[theorem]{Definition}
\newtheorem{example}[theorem]{Example}

\theoremstyle{remark}
\newtheorem{remark}[theorem]{Remark}

\let\emph\relax 
\DeclareTextFontCommand{\emph}{\bfseries\em}
\DeclareMathOperator{\rank}{rank}
\DeclareMathOperator{\Skew}{Skew}

\tikzstyle{edge}=[line width=1.5pt,black]
\tikzstyle{vertex}=[fill=black,circle,inner sep=0pt, minimum size=5.5pt]

\title{Extremal decompositions of tropical varieties and relations with rigidity theory}
\author{Farhad Babaee\thanks{School of Mathematics, University of Bristol, \texttt{farhad.babaee@bristol.ac.uk}} \and Sean Dewar\thanks{School of Mathematics, University of Bristol, \texttt{sean.dewar@bristol.ac.uk}} \and James Maxwell\thanks{School of Mathematics, University of Bristol, \texttt{james.maxwell@bristol.ac.uk}}
}

\begin{document}
\date{}
\maketitle

\begin{abstract}

Extremality and irreducibility constitute fundamental concepts in mathematics, particularly within tropical geometry. While extremal decomposition is typically computationally hard, this article presents a fast algorithm for identifying the extremal decomposition of tropical varieties with rational balanced weightings. Additionally, we explore connections and applications related to rigidity theory.    
In particular,
we prove that a tropical hypersurface is extremal if and only if it has a unique reciprocal diagram up to homothety.
\end{abstract}




\section{Introduction}
Let $V$ be a vector space over the field $\F \in \{\Q, \R \},$ and $K\subseteq V$ be a closed convex set. Recall that an element $v\in K$ is called \emph{extremal} if any decomposition $ v = v_1+  v_2$ for $v_1, v_2 \in K$ implies that there are non-negative scalars $\lambda_1, \lambda_2 \in \F$ such that $v= \lambda_1 v_1$ and $v= \lambda_2 v_2. $ As it turns out, any element of a compact convex set can be written as a (limit of) linear combinations of extremal elements with positive coefficients. More concretely, the Krein-Milman theorem (see \cite[Theorem 3.23]{Rudin}) -- or more generally Choquet's theorem (see \cite{Phelps}) --  imply that when $K$ is a compact subset of a Hausdorff locally convex topological vector space $V,$ then the closure of the convex hull of the set of extremal points of $K$ coincides with $K$. In consequence, understanding the extremal elements of different convex sets is a fundamental question in mathematics. For instance: 

\begin{enumerate}
    \item \label{firstpoint} In algebraic geometry, \emph{irreducible algebraic varieties} are building blocks of algebraic varieties, and in the following sense they relate to extremality: Let $X$ be a smooth algebraic variety, then the extremal elements of the set of \emph{effective $\F$-cycles} 
    $$ \left\{\sum \lambda_i [Z_i]: Z_i \text{ an algebraic subvariety of pure dimension $p$},\, \lambda_i \in  \F_{\geq 0}\right\}$$
    are of the form $\lambda [Z],$ where $\lambda \geq 0$ and $Z$ is irreducible. 
    
    \item \label{secondpoint} In ergodic theory and dynamical systems, given a probability space  $(X, B, \mu )$ and a map $T: X \to X,$ the ergodic measures are exactly the extremal elements in the convex set of $T$-invariant probability measures; see \cite[Proposition 12.4]{Phelps}.

    \item \label{thirdpoint} In tropical geometry, one can define 
    \emph{extremal tropical varieties}; see \Cref{sec:trop_var}. 
    Interestingly, certain tropical varieties called \emph{Bergman fans}, which are cryptomorphic to {matroids}, are extremal; see \cite{Huh}.

    \item  In analytic geometry, one can consider the space of \emph{currents} on a complex manifold, which is the set of continuous functionals on smooth forms with compact support. In this space, one defines the cone of positive currents. It turns out that the set of the extremal elements of this cone contains the currents associated to extremal examples given in \ref{firstpoint} and \ref{thirdpoint} above; see \cite{Lelong} and \cite{BH}, respectively. Invariant extremal currents extend the notion of ergodicity in \ref{secondpoint}. Currents associated to tropical extremal varieties also gave rise to a family of counter-examples to the generalised Hodge conjecture for positive currents; see \cite{BH, Adi-Baba}.
\end{enumerate}

The above examples justify why tropical extremal varieties are significant, and extremal decompositions are useful. Note that in a cone which is a positive span of finitely many elements or a polytope which is a convex hull of finitely many points in $\R^n$, the set of extremal elements are finite, and in consequence (\Cref{t:allirreducible}), extremal decompositions of tropical varieties are always finite.

In recent years,
tropical geometry has been applied with great success to the area of \emph{rigidity theory} -- the study of kinematics for bar-and-joint frameworks.
See for example \cite{bernstein2019tropical,cggkls,gls}.
There has, however, not been any such research seeking to explore how concepts in rigidity theory can be applied to tropical geometry.
In this paper we explore how the language of rigidity theory can be applied to better understand extremality for tropical varieties.
The key link that achieves this is the observation that balanced weightings for tropical varieties play an identical role to \emph{equilibrium stresses} for frameworks -- a physical quantity that measures the internally-generated forces of an over-constrained framework.

We begin by restricting to the case of \emph{tropical hypersurfaces} --
tropical varieties determined by a single tropical polynomial.
Since each maximal face is now a $(d-1)$-dimensional convex polytope,
we are able to construct what is known as a \emph{reciprocal diagram} (also known as a \emph{Maxwell reciprocal diagram}).
This is a pair $(G,p)$,
where $G$ is the dual graph of the tropical hypersurface (a vertex for each connected subset of the complement and an edge between vertices whose corresponding connected components share a maximal face) and straight-edge graph embedding $p$ where each edge is perpendicular to its corresponding maximal face in the tropical hypersurface (see \Cref{def:dualframework} for more details).
In particular,
the subdivided Newton polytope of a tropical polynomial $f$ is always a reciprocal diagram of the tropical hypersurface defined by $f$.
The use of reciprocal diagrams to investigate static properties of framework structures was first initiated by James Clerk Maxwell\footnote{No relation to the third author.} \cite{maxwell45}, and these techniques are still used to this very day \cite{sm23,baker}.
In the context of planar bar-and-joint frameworks,
a framework has a unique equilibrium stress if and only if a chosen reciprocal diagram has exactly one \emph{parallel redrawing} up to homothety:
any other embedding of the dual graph with all edges parallel to the original reciprocal diagram.
In \Cref{sec:hyper},
we prove that an analogous statement is also true for tropical hypersurfaces.

\begin{theorem}\label{t:main}
    Let $C$ be a tropical hypersurface in $\mathbb{R}^d$,
    and let $(G,p)$ be a reciprocal diagram of $C$.
    Then the following properties are equivalent.
    \begin{enumerate}
        \item \label{t:main1} $C$ is extremal.
        \item \label{t:main3} $(G,p)$ is \emph{direction rigid}, i.e.,
        if $(G,q)$ is a framework in $\mathbb{R}^d$ where each edge of $(G,q)$ is parallel to its corresponding edge in $(G,p)$,
        then $(G,q)$ is a scaled and translated copy of $(G,p)$.
    \end{enumerate}
\end{theorem}

\Cref{t:main} indicates that extremality is a dual concept to ``rigidity'' for tropical hypersurfaces.
This concept is more clearly seen for tropical curves (tropical hypersurfaces in the plane).
Here we apply well-known structural engineering techniques (see \Cref{subsec:parallelredraw}) to obtain the following result.

\begin{corollary}\label{cor:main}
    Let $C$ be a tropical curve in $\R^2$ and let $(G,p)$ be a reciprocal diagram of $C$.
    Then the following properties are equivalent.
    \begin{enumerate}
        \item \label{cor:main1} $C$ is extremal.
        \item \label{cor:main4} $(G,p)$ is infinitesimally rigid.
    \end{enumerate}
\end{corollary}
\noindent
Combinatorial consequences to \Cref{cor:main} are explored in greater detail in \Cref{sec:comb}.

In \Cref{sec:decomp}, we explore extremality in general tropical varieties.
We do so by introducing a ``rigidity matrix'' $R(C)$ for a given tropical variety $C$ (see \Cref{subsec:rc} for a detailed guide on how to construct $R(C)$).
The importance of the matrix $R(C)$ is that a weighting for a tropical variety (possibly with irrational and negative entries) satisfies the balancing condition if and only if it lies in the left kernel of $R(C)$.
This mirrors the situation of equilibrium stresses for bar-and-joint frameworks,
where an edge weighting is an equilibrium stress if and only if it is an element of the left kernel of the framework's \emph{rigidity matrix} (see \Cref{subsec:rigid}).
Using this rigidity matrix, we are able to prove the following results:

\begin{theorem}\label{t:detect}
    Let $C$ be a $k$-dimensional tropical variety in $\mathbb{R}^d$.
    Then the largest integer $n$ such that $C$ has $n$ linearly independent balanced weightings is equal to $\dim \ker R(C)^T$.
    Hence $C$ is extremal if and only if $\rank R(C) = |\widetilde{E}|-1$.
\end{theorem}

We conclude the paper by investigating extremal decompositions for tropical varieties,
i.e., if $C$ is a tropical variety,
find a set of extremal tropical subvarieties of $C$ that cover $C$.
We determine two methods for identifying all such decompositions:
either by finding minimal generating elements of a positive cone (\Cref{t:allirreducible}),
or by finding sets of vertices of a polytope which contain an interior point of the polytope in their convex hull (\Cref{t:decompperfect}).
Both of these results, and many others,
allow us to construct an efficient algorithm for extremal decompositions.

\begin{theorem}
    There exists a fast algorithm (as described in \Cref{subsec:alg}) for constructing an extremal decomposition of any given tropical variety. 
\end{theorem}

\section{Preliminaries}
We begin by introducing the necessary background from both the areas of tropical geometry (\Cref{sec:trop_var}) and rigidity theory (\Cref{subsec:rigid,subsec:parallelredraw}). Tropical geometry is the piece-wise linear counterpart of algebraic geometry where polyhedral complexes play the role of varieties. Whereas, rigidity theory studies when and how bar-and-joint frameworks can be deformed. 
\subsection{Tropical varieties}\label{sec:trop_var}

Let $\mathbb{T} = \mathbb{R} \cup \{-\infty\}$ be the tropical semiring with the tropical addition and multiplication operations:
\begin{align*}
    x \oplus y &:= \max \{x,y\},\\
    x \otimes y &:= x +y.
\end{align*}
The identity for tropical addition is $-\infty$, while the identity for tropical multiplication is $0$.
Note that while every element $x$ has a tropical multiplication inverse $-x$,
they do not have a tropical addition inverse.
For any positive integer $k$, we define $x^{\otimes k}$ to be the tropical multiplication of $k$ copies of $x$, $x^{\otimes (-k)}$ to be the tropical multiplication of $k$ copies of $-x$, and $x^{\otimes 0} = 0$.

Given an element $z = (z_1,\ldots,z_d) \in \mathbb{T}^d$ and an element $k = (k_1,\ldots,k_d) \in \mathbb{Z}^d$,
we define
\begin{align*}
    z^{\otimes k} := z_1^{\otimes k_1} \cdots z_d^{\otimes k_d} = k \cdot z,
\end{align*}
where $\cdot$ is the standard inner product.
A \emph{tropical (Laurent) polynomial} is a map $f : \mathbb{T}^d \rightarrow \mathbb{T}$ where there exists a finite set $A \subset \mathbb{Z}^d$ and a set $(a_k)_{k \in A} \in \mathbb{R}^A$ such that
\begin{align*}
    f(z) := \bigoplus_{k \in A} a_k \otimes z^{\otimes k} = \max \{ k \cdot z + a_k : k \in A\}.
\end{align*}
The \emph{tropical hypersurface} of a tropical polynomial $f$ is the set
\begin{align*}
    \mathbb{V}(f) := \left\{z \in \mathbb{R}^d : f \text{ is not differentiable at } z\right\}.
\end{align*}
Equivalently,
$\mathbb{V}(f)$ is the set of points $z \in \mathbb{R}^d$ where $f(z) = a_k \otimes z^{\otimes k} = a_\ell \otimes z^{\otimes \ell}$ for distinct $k,\ell \in A$. 
This idea of tropical vanishing can be extended to collections polynomials in a natural way.

Every tropical hypersurface is a polyhedral object that is part of a larger class of objects called \emph{tropical varieties}, which are defined as follows.
Let $C \subset \mathbb{R}^d$ be a polyhedral complex in $\mathbb{R}^d$ of dimension $k$.
We fix $\widetilde{E}$ to be the set of polytopes of dimension $k$ in $C$ (said to be the \emph{maximal faces} of $C$) and $\widetilde{V}$ to be the set of polytopes of dimension $k-1$ in $C$ (said to be the \emph{ridges} of $C$). 
Throughout the paper, we always assume that our polyhedral complexes are 
\emph{pure} (every polytope with dimension $k-1$ or less is a face of polytope of maximal dimension) and \emph{rational} (given $L_\sigma$ is the $k$-dimensional linear space parallel to $\sigma \in \widetilde{E}$, the set $L_\sigma \cap \mathbb{Z}^d$ is a rank $k$ lattice).
A map $\omega : \widetilde{E} \rightarrow S$ is said to be a \emph{rational partial weighting} if $S= \mathbb{Q}_{\geq 0}$, a \emph{rational weighting} if $S= \mathbb{Q}_{>0}$, a \emph{partial weighting} if $S= \mathbb{Z}_{\geq 0}$, and a \emph{weighting} if $S= \mathbb{Z}_{>0}$.
Any (rational) partial weighting that is not a (rational) weighting (i.e., it takes the value $\omega(\sigma)=0$ for some maximal face $\sigma$) is said to be a (\emph{rational} resp.) \emph{strictly partial weighting}.
It is important to note that every rational partial weighting can be scaled by some positive integer to form a partial weighting.

We now wish to describe the concept of balanced weightings. 
For each $\tau \in \widetilde{V}$,
fix $L_\tau$ to be the $(k-1)$-dimensional linear space parallel to $\tau$.
It is immediate that if $\sigma \in \widetilde{E}$ contains $\tau$ then $L_\tau \subset L_\sigma$.
Since $C$ is rational,
we have that $L_\tau \cap \mathbb{Z}^d$ is a rank $(k-1)$ lattice.
For every $\sigma \in \widetilde{E}$ and every $\tau \in \widetilde{V}$ with $\tau \subset \sigma$,
we now fix the unique vector $z_{\tau}(\sigma) \in L_\sigma \cap \mathbb{Z}^d$ such that $L_\sigma \cap \mathbb{Z}^d = L_\tau \cap \mathbb{Z}^d + \mathbb{Z} z_{\tau}(\sigma)$ and $x + \lambda z_{\tau}(\sigma) \in \sigma$ for some $x \in \tau$ and some $\lambda >0$.
We now say that a (rational, partial) weighting $\omega$ of $C$ is \emph{balanced} if for every $\tau \in \widetilde{V}$ we have
\begin{align*}
    \sum_{\sigma \supset \tau} \omega(\sigma) z_{\tau}(\sigma) \in L_\tau,
\end{align*}
where the sum runs over all $\sigma \in \widetilde{E}$ that share $\tau$ as a face.
If $k$-dimensional polyhedral complex $C$ has a balanced weighting $\omega :\widetilde{E} \rightarrow \mathbb{Z}_{>0}$, then it is said to be a \emph{tropical variety of dimension $k$}. Under this definition, tropical varieties with codimension $1$ and tropical hypersurfaces are equivalent; see, for example, \cite[Theorem 3.3.5, Proposition 3.3.10]{Mac+Stur:15}.

\begin{remark}
    This differs from the definition given by Maclagan and Sturmfels \cite{Mac+Stur:15}, where a tropical variety is the tropicalisation of an ideal of a Laurent polynomial ring over a valuated field. 
    Maclagan and Sturmfels' definition for a tropical variety is equivalent to our own for tropical hypersurfaces,
    but is otherwise a strictly stronger definition;
    see, for example, \cite[Theorem 3.3.5, Example 4.2.15]{Mac+Stur:15}.
\end{remark}

When discussing tropical varieties, there always exists more than one balanced weighting,
since any balanced weighting can be scaled by a positive integer and remain balanced.
A tropical variety $C$ with a fixed balanced weighting $\omega$ is said to be a \emph{weighted tropical variety},
which we denote by $(C,\omega)$. 

\begin{figure}
        \centering
        \begin{subfigure}[h]{0.45\textwidth}
            \centering
            \includegraphics{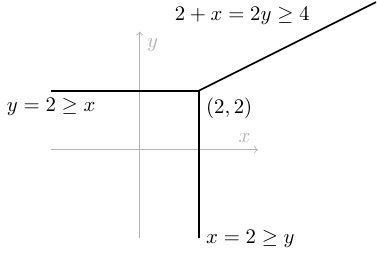}
        \caption{The tropical variety of $f = 2 \otimes x \oplus y^{\otimes 2} \oplus 4$}
        \label{Fig:exs-trop-var-line}
        \end{subfigure}
        \begin{subfigure}[h]{0.45\textwidth}
            \centering
            \includegraphics{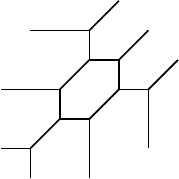}
        \caption{A tropical cubic variety of genus 1}
        \label{Fig:exs-trop-var-cubic}
        \end{subfigure}
        \caption{Examples of tropical varieties from \Cref{Ex:tropical-varieties}}
        \label{Fig:exs-trop-var}
    \end{figure}

\begin{example}\label{Ex:tropical-varieties}
Here are several examples of tropical varieties, including the standard tropical line in $\mathbb{T}^2$. They are depicted in \Cref{Fig:exs-trop-var,fig:tropical_hyperplane_3D}.
    \begin{itemize}
        \item Take the polynomial $f \in \mathbb{T}[x,y]$, defined as $f = 2 \otimes x \oplus y^{\otimes 2} \oplus 4 = \max \{ 2 +x , 2y,4\}$. Then the tropical hypersurface, depicted in \Cref{Fig:exs-trop-var-line}, is defined the  following three equations: 
        \[
        x+2=2y\ge4 , \quad x= 2\ge y , \quad y= 2 \ge x.
        \]
        \item \Cref{Fig:exs-trop-var-cubic} is the tropical variety of a cubic polynomial described in \cite[Example 3.1.8 (3)]{Mac+Stur:15}. Note that it has genus one, and three infinite rays in the west, south and north east directions. 
        \item \Cref{fig:tropical_hyperplane_3D} shows a 2-dimensional tropical in $\mathbb{R}^3$ cut out by the tropical polynomial $f = 0 \oplus x \oplus y \oplus z $. For other visual depictions of higher dimensional tropical varieties see \cite{Mikhalkin+Rau:19}.
    \end{itemize}
\end{example}

\begin{figure}[htp]
        \centering
        \includegraphics[scale=0.75]{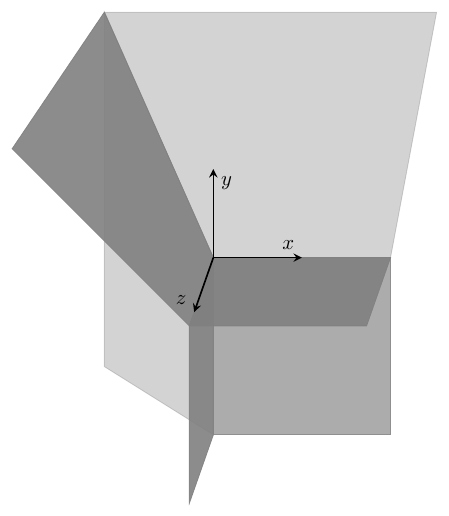}
        \caption{The tropical variety $\mathbb{V}(0 \oplus x \oplus y \oplus z)$ }
        \label{fig:tropical_hyperplane_3D}
    \end{figure}

Let $C \subset \mathbb{R}^d$ be a tropical variety of dimension $k$ with a balanced partial weighting $\omega$ with support $S \not\subset \widetilde{E}$.
Let $\sim_*$ be the symmetric relation on the set $S$ where $\sigma \sim_* \sigma'$ if and only if $\sigma = \sigma'$ or $\tau = \sigma \cap \sigma' \in \widetilde{V}$ and $\tau$ is not contained in any $k$-dimensional polytope contained in the set $S \setminus \{\sigma,\sigma'\}$.
From this, we fix $\sim$ to be the equivalence relation formed from the transitive closure of $\sim_*$.
For any $\sigma \in S$,
the set $\sigma_\sim := \{\sigma' \in S : \sigma' \sim \sigma \}$ is a $k$-dimensional polytope.
Furthermore,
the set $C' := \bigcup_{\sigma \in S} \sigma$ is a tropical variety of dimension $k$ with $k$-dimensional polytopes $\widetilde{E}' := S /\sim$ and balanced weighting $\omega' : \widetilde{E}' \rightarrow \mathbb{Z}_{>0}$ such that $\omega'(\sigma_\sim) = \omega(\sigma)$.
We say that $\omega'$ is the \emph{refinement} of $\omega$,
and $\omega$ is the \emph{coarsening} of $\omega'$.

For two any polyhedral complexes $C_1$ and $C_2$ in $\mathbb{R}^d$ of dimension $k$, 
the set $C = C_1 \cup C_2$, after a suitable refinement is also a polyhedral complex of dimension $k$.
Furthermore, if $C_1$ and $C_2$ are tropical varieties,
then $C$ is also a tropical variety.
To see this,
choose balanced weightings $\omega_1$ and $\omega_2$ for $C_1$ and $C_2$ respectively.
Since $C_1$ and $C_2$ cover $C$,
it follows that any positive linear combination of the extensions of $\omega_1$ and $\omega_2$ is a balanced weighting of $C$.
With this in mind,
we define the following property.

\begin{definition}
A tropical variety $C \subset \mathbb{R}^d$ is \emph{extremal} if it cannot be decomposed into tropical varieties $C_1 ,C_2 \subset \mathbb{R}^d$, both proper subsets of $C$, such that $C = C_1 \cup C_2$.
\end{definition}

Equivalent definitions of extremal tropical varieties are provided by the following result.

\begin{proposition}\label{lem:extre-decomp}
    Let $(C, \omega)$ be a weighted tropical variety in $\mathbb{R}^d$.
    Then the following properties are equivalent:
    \begin{enumerate}
         \item \label{lem:extre-decomp1} $C$ is extremal;
        \item \label{lem:extre-decomp2} for every pair of weighted tropical varieties $(C_1, \omega_1) $, $(C_2, \omega_2)$ in $\mathbb{R}^d$ such that $C = C_1 \cup C_2$ and (after common refinement) $\lambda\omega = \omega_1 + \omega_2$ for some positive integer $\lambda$, there exists rational scalars $\lambda_1,\lambda_2$ such that $(C, \omega) =   (C_i,  \lambda_i \omega_i)$ for $i=1,2$;
        \item \label{lem:extre-decomp3} $\omega$ is the unique balanced weighting of $C$ up to rational scalar multiplication;
        \item \label{lem:extre-decomp4} 
        $C$ contains no proper subset that is also a tropical variety of the same dimension.
    \end{enumerate}
\end{proposition}

\begin{remark}
Let us highlight that the question of extremal decomposition becomes computationally much more difficult when we deal with positive integer weights.
To showcase this,
we define the following property for a weighted tropical variety $(C,\omega)$:
\begin{itemize}
    \item[($\ast$)] For every pair of weighted tropical varieties $(C_1, \omega_1) $, $(C_2, \omega_2)$ in $\mathbb{R}^d$ such that $C = C_1 \cup C_2$ and (after common refinement) $\omega = \omega_1 + \omega_2$, there exists positive integers $\lambda_1,\lambda_2$ such that $(C, \omega) =   (C_i,  \lambda_i \omega_i)$ for $i=1,2$.
\end{itemize}
It is important here to note that if we allow $\omega_1,\omega_2$ to be rational weightings and we allow $\lambda_1,\lambda_2$ to be rational, property ($\ast$) becomes equivalent to property \ref{lem:extre-decomp2} of \Cref{lem:extre-decomp}.
However,
property ($\ast$) is a stronger condition than property \ref{lem:extre-decomp2} of \Cref{lem:extre-decomp}.
For an example of this,
observe the weighted tropical variety given in \Cref{fig:Red-Trop-Var-Bad-W}.
\begin{figure}[ht]
    \centering
    \includegraphics[width= 0.5 \textwidth]{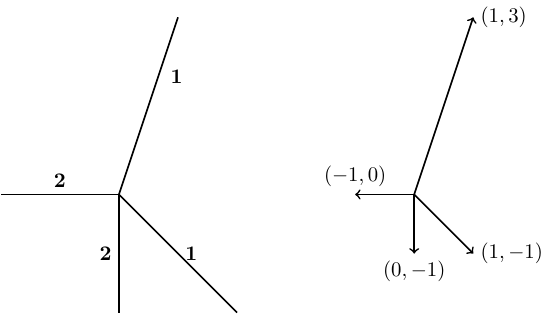}
    \caption{A weighted tropical variety which can not be decomposed.} 
    \label{fig:Red-Trop-Var-Bad-W}
\end{figure}

\noindent
While the tropical variety in \Cref{fig:Red-Trop-Var-Bad-W} can be decomposed into proper tropical subvarieties (and so does not satisfy property \ref{lem:extre-decomp2} of \Cref{lem:extre-decomp}),
there is no decomposition that respects its fixed balanced weighting (and so the tropical variety does satisfy property ($\ast$)).
Property ($\ast$) does fail to hold, however, if the balanced weighting given in \Cref{fig:Red-Trop-Var-Bad-W} is scaled by $3$; see \Cref{fig:Red-Trop-Var-Decomp}.
\begin{figure}[ht]
    \centering
    \includegraphics[width=0.65\textwidth]{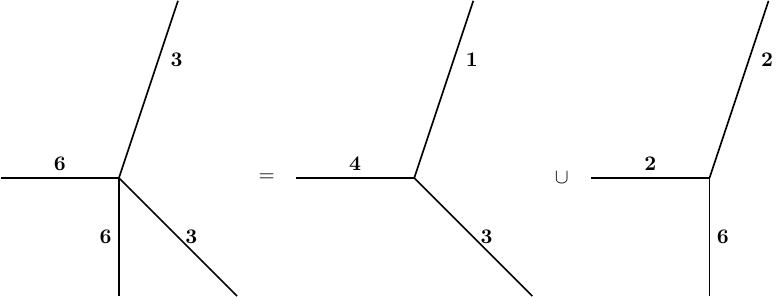}
    \caption{The tropical variety from \Cref{fig:Red-Trop-Var-Bad-W} with weighting scaled by 3 and the corresponding decomposition.}
    \label{fig:Red-Trop-Var-Decomp}
\end{figure}

We note here that property ($\ast$) is closely related to decomposing tropical polynomials.
In \cite{crowell2019tropical},
Crowell observed that a tropical polynomial $f$ has no decomposition $f = g \otimes h$ if and only if property ($\ast$) holds for the tropical hypersurface $(\mathbb{V}(f),\omega)$,
where $\omega$ is the induced balanced weighting on $\mathbb{V}(f)$ (see \Cref{p:naturalweighting} for the precise definition). 
\end{remark}

\begin{proof}[Proof of \Cref{lem:extre-decomp} ]
\ref{lem:extre-decomp1} $\implies$ \ref{lem:extre-decomp2}:
Suppose that \ref{lem:extre-decomp2} does not hold,
and fix $(C_1,\omega_1)$ and $(C_2,\omega_2)$ to be witnesses to this.
If both $C_1,C_2$ are proper then $C$ is not extremal.
Suppose that $C_1 = C$ (and thus $C_2 \subset C$).
We note here that the common refinement of $\omega_2$ is linearly independent of $\omega_1$ whether or not $C_2 = C$.
For sufficiently large $t>0$ the vector $t \omega_1 - \omega_2$ has only positive coordinates and for sufficiently small $t>0$ the vector $t \omega_1 - \omega_2$ has some negative coordinates.
Hence,
there exists $\mu_1,\mu_2 > 0$ so that $\omega_3 := \mu_1 \omega_1 - \mu_2 \omega_2$ is a balanced strictly partial weighting of $C$ (here we have chosen $t= \mu_1/\mu_2$ to achieve a rational strictly partial weighting and then scaled by $\mu_2$ to obtain integer weights for each maximal face).
Fix $C_3$ to be the proper tropical variety in $C$ which is the support of $\omega_3$.
If $C_2 \neq C$ then we observe from the zeroes of $\omega_2,\omega_3$ that $C=C_2 \cup C_3$ as required.
If $C_2 = C$,
then we replace $C_2$ with $C_3$ and repeat the above argument to obtain the desired decomposition.

\ref{lem:extre-decomp2} $\implies$ \ref{lem:extre-decomp3}:
Suppose that \ref{lem:extre-decomp2} holds.
Let $\omega'$ be a balanced weighting of $C$.
Choose any positive integer $\mu$ such that $\omega_1 := \mu \omega - \omega'$ is a balanced weighting of $C$.
By fixing $C_1=C_2=C$ and $\omega_2 = \omega'$,
we have that $\omega = \omega_1 + \omega_2$ and $C = C_1 \cup C_2$.
Hence, both $\omega_1,\omega_2$ are scaled copies of $\omega$.
It is now immediate that $\omega' = \lambda \omega$ for some rational scalar $\lambda$.

\ref{lem:extre-decomp3} $\implies$ \ref{lem:extre-decomp4}:
Suppose that $C$ contains a proper tropical variety $C'$ of the same dimension.
Fix $\rho$ to be a balanced weighting of $C'$,
and fix $\rho'$ to be its extension to a strictly partial balanced weighting of $C$.
As $\rho'$ has zero coordinates and $\omega$ does not,
they are linearly independent.
Hence $\omega + \rho'$ is a balanced weighting of $C$ that is linearly independent of $\omega$.

\ref{lem:extre-decomp4} $\implies$ \ref{lem:extre-decomp1}:
This is immediate.
\end{proof}  

\begin{example}\label{Ex:extremal-trop-var}
    The tropical variety in \Cref{Fig:extremal-trop-quadratic} is extremal. A weighting $\omega$ must be balanced with respect to the primitive integer vectors
    \begin{align*}
        P_1 = \{(1,1),(-1,0),(0,-1)\}, \qquad P_2=\{(-1,-1),(1,0),(0,1)\}.   
    \end{align*}
    As $P_1 = -P_2$ we only need to describe $P_1$ and multiply weights by $-1$ for $P_2$. Denote by $\{\sigma_1, \sigma_2, \sigma_3\}$ the codimension one cones corresponding to the vectors $\{(1,1),(-1,0),(0,-1)\}$, as per the figure. For a weighting to be balanced it must satisfy $\omega(\sigma_1)(1,1) + \omega(\sigma_2)(-1,0) +\omega(\sigma_3)(0,-1) = (0,0)$. Though, it can be seen that this implies $\omega(\sigma_1) = \omega(\sigma_2) = \omega(\sigma_3)$. Therefore, there is a unique weighting of the variety up to scaling, and hence extremal.   
    \begin{figure}
        \centering
        \includegraphics{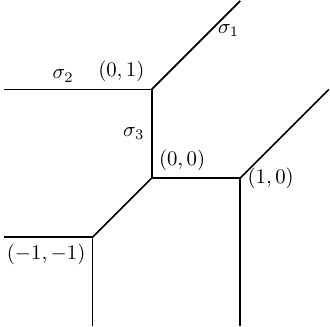}
        \caption{An extremal tropical variety}
        \label{Fig:extremal-trop-quadratic}
    \end{figure}
\end{example}

\begin{example}
In \Cref{Fig:Reducible-Trop-Var} there is a description of a reducible tropical variety and two extremal decompositions. The weightings for each variety are not noted in the figure, but one suitable weighting would be all weights equal to one.
\begin{figure}
    \centering
    \includegraphics[width=\textwidth]{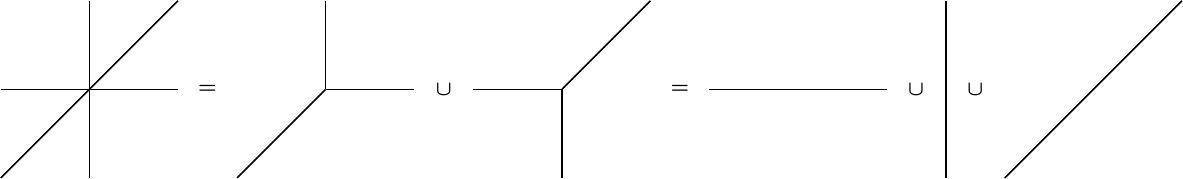}
    \caption{A reducible tropical variety and two extremal decompositions.}
    \label{Fig:Reducible-Trop-Var}
\end{figure}
See \cite[Example 2.7]{Mac+Rin:18} for more details of this variety, where it is utilised to illustrate that tropical ideals carry strictly more information than their tropical varieties. 
\end{example}

\subsection{Rigidity and infinitesimal rigidity of frameworks}\label{subsec:rigid}

Given a (finite simple) graph $G=(V,E)$,
a \emph{$d$-dimensional realisation} of $G$ is a map $p:V \rightarrow \mathbb{R}^d$.
We say that a graph-realisation pair $(G,p)$ is a \emph{framework} in $\mathbb{R}^d$.
Two frameworks $(G,p)$ and $(G,q)$ in $\mathbb{R}^d$ are said to be \emph{equivalent} if for each edge $uv \in E$ we have
\begin{align}\label{eq:equiv}
    \|p(u)-p(v)\| = \|q(u)-q(v)\|.
\end{align}
The frameworks $(G,p),(G,q)$ are said to be \emph{congruent} if \cref{eq:equiv} holds for any pair of vertices $u,v \in V$.
A framework $(G,p)$ is now said to be \emph{rigid} if there exists $\varepsilon >0$ such that every equivalent framework $(G,q)$ with $\|p(v)-q(v)\| < \varepsilon$ for each $v \in V$ is also congruent to $(G,p)$.

Determining whether a framework is rigid is NP-Hard when $d \geq 2$ \cite{saxe1980embeddability}.
To proceed we employ a stronger notion of rigidity which is more tractable using the \emph{rigidity matrix}. This is 
the Jacobian matrix of the quadratic congruent relations describing each edge of the framework.
To be specific, the rigidity matrix of a framework $(G,p)$ is the $|E| \times d|V|$ matrix $R(G,p)$ with the row labelled $uv \in E$ given by
\begin{align*}
	\Big[ \quad 0 \quad \cdots  \quad 0  \quad \overbrace{(p(u)-p(v))^\top}^{u} \quad 0  \quad \cdots  \quad 0  \quad \overbrace{(p(v)-p(u))^\top}^{v} \quad 0  \quad \cdots  \quad 0 \quad \Big].
\end{align*}
It is immediate that $\rank R(G,p) \leq |E|$.
It is less obvious that we also have that $\rank R(G,p) \leq d|V| - \binom{d+1}{2},$ so long as $(G,p)$ has affine dimension $d$, i.e., the affine span of the set $\{ p(v) :v \in V\}$ has dimension $d$.
This is a consequence of the set of frameworks equivalent to $(G,p)$ being invariant under the group of isometries of $\mathbb{R}^d$,
which in turn implies the tangent space of the isometry group at the identity map is a subspace of the kernel of $R(G,p)$. 
To see this, we first note that the tangent space of the isometry group at the identity is exactly the direct product of the skew-symmetric matrices and constant vector maps. Second, we define  $u_i := (e_i)_{v \in V}$ for each $i \in \{1,\ldots,d\}$ (with $e_1,\ldots,e_d$ being a basis of $\mathbb{R}^d$) and $\Skew_{d \times d}$ to be the linear space of skew-symmetric $d \times d$ matrices.
From this, we define the linear spaces
\begin{align*}
    R := \Big\{(M(p(v)))_{v \in V} : M \in \Skew_{d \times d} \Big\}, \qquad T:= \left\{ \sum_{i=1}^d a_i u_i : a_1,\ldots,a_d \in \mathbb{R} \right\}.
\end{align*}
It is clear that $T$ has dimension $d$ and is contained in $\ker R(G,p)$.
Since every skew-symmetric matrix $M$ has the property that $x^\top M x = 0$,
the space $R$ is contained in $\ker R(G,p)$ also.
If $(G,p)$ also has affine dimension $d$ then $R \cap T = \{0\}$ and $\dim R = \binom{d}{2}$,
and hence
\begin{align*}
    \dim \ker R(G,p) \geq \dim R + \dim T = \binom{d}{2} + d = \binom{d+1}{2}.
\end{align*}
Any element of $\ker R(G,p)$ is said to be an \emph{infinitesimal flex},
and any element of $R + T$ is said to be a \emph{trivial infinitesimal flex}.
We define a framework $(G,p)$ to be \emph{infinitesimally rigid} if every infinitesimal flex of $(G,p)$ is trivial.
Equivalently,
$(G,p)$ is infinitesimally rigid if and only if either $\rank R(G,p) = d|V| - \binom{d+1}{2}$, or $(G,p)$ is a simplex (i.e., $G$ is a complete graph with $k \leq d+1$ vertices and $(G,p)$ has affine dimension $k-1$).
Any framework that is not infinitesimally rigid is said to be \emph{infinitesimally flexible}.
We refer an interested reader to \cite[Section 2]{GraverServatius} for more details about infinitesimal rigidity and the claims made within this section.

Infinitesimal rigidity is much easier to check for than rigidity since it (usually) requires checking the rank of an easily computable matrix.
It is (usually) a sufficient condition for rigidity also.

\begin{theorem}[\cite{AsimowRothI,AsimowRothII}]\label{t:asimowroth}
    If a framework $(G,p)$ is infinitesimally rigid then it is rigid.
    If $(G,p)$ is rigid and has maximal rank -- i.e., $\rank R(G,p) \geq \rank R(G,q)$ for all other choices of $q \in (\mathbb{R}^d)^V$ -- then it is infinitesimally rigid.
\end{theorem}

The maximal rank condition of \Cref{t:asimowroth} unfortunately cannot be removed;
see \Cref{fig:prism} for an example of a framework in $\mathbb{R}^2$ that is rigid but (since it does not have maximal rank rigidity matrix) is not infinitesimally rigid.

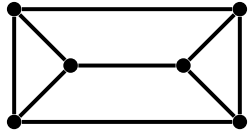
\begin{figure}[htp]
	\begin{center}
        \begin{tikzpicture}[scale=0.75]
			\node[vertex] (11) at (-1,0) {};
			\node[vertex] (21) at (-2,1) {};
			\node[vertex] (31) at (-2,-1) {};
			
			\node[vertex] (12) at (1,0) {};
			\node[vertex] (22) at (2,1) {};
			\node[vertex] (32) at (2,-1) {};
			
			\draw[edge] (11)edge(12);
			\draw[edge] (21)edge(22);
			\draw[edge] (31)edge(32);
			
			\draw[edge] (11)edge(21);
			\draw[edge] (21)edge(31);
			\draw[edge] (31)edge(11);
			
			\draw[edge] (12)edge(22);
			\draw[edge] (22)edge(32);
			\draw[edge] (32)edge(12);
		\end{tikzpicture}
	\end{center}
	\caption{A framework that is rigid but not infinitesimally rigid.}\label{fig:prism}
\end{figure}

It should be noted that infinitesimal rigidity is a generic property,
in that, given two generic\footnote{Here generic will mean that the coordinates of the realisation $p$ (when considered as a vector in $\mathbb{R}^{d|V|}$) will form an algebraically independent set of $d|V|$ elements.} frameworks $(G,p)$ and $(G,q)$ in $\mathbb{R}^d$,
one will be infinitesimally rigid if and only if the other is infinitesimally rigid.
Hence the set of infinitesimally rigid $d$-dimensional realisations of a graph will either be an \text{orange}{the} empty subset of $(\mathbb{R}^d)^V$ (in which case we say that the graph is \emph{flexible in $\mathbb{R}^d$}),
or will have Lebesgue measure zero complement (in which case we say that the graph is \emph{rigid in $\mathbb{R}^d$}).

It is natural to ask what graph properties are necessary and/or sufficient for graph rigidity/flexibility.
With this in mind, we proceed with the following concept.
For a graph $G=(V,E)$ and any subset $X \subset V$,
we define $i_G(X)$ to be the number of edges in the subgraph of $G$ induced by the vertex set $X$.
Given non-negative integers $d,k$ with $k \in \{d+1,\binom{d+1}{2}\}$,
we say that a graph $G$ is \emph{$(d,k)$-sparse} if $i_G(X) \leq d|X| - k$ for all $|X| \geq d + 2$.
We say that $G$ is \emph{$(d,k)$-tight} if it is $(d,k))$-sparse and $|E|=d|V|-k$.

\begin{theorem}[\cite{maxwell}]\label{t:maxwell}
    Let $G=(V,E)$ be a rigid graph in $\mathbb{R}^d$ with at least $d$ vertices.
    Then $G$ contains a spanning $(d,\binom{d+1}{2})$-tight subgraph $H$ that is also rigid in $\mathbb{R}^d$.
    Hence $|E|=d|V|-\binom{d+1}{2}$.
\end{theorem}

Since $(1,1)$-tight graphs are trees and generic rigidity in $\mathbb{R}^1$ is equivalent to connectivity, the converse of \Cref{t:maxwell} holds for $d=1$.
As was proven by Pollaczek-Geiringer \cite{HPG} (and later rediscovered by Laman \cite{laman}),
the converse of \Cref{t:maxwell} holds for $d=2$.

\begin{theorem}\label{t:laman}
    A graph is rigid in $\mathbb{R}^2$ if and only if either it is a single vertex or it contains a $(2,3)$-tight\footnote{$(2,3)$-tight graphs are also known as \emph{Laman graphs} in various literature.} spanning subgraph.
\end{theorem}

This is, however, not true when $d \geq 3$.
For example,
the graph in \Cref{fig:db} (known as the double-banana graph) is $(3,6)$-tight but very clearly flexible in $\mathbb{R}^3$ since it has a separating set of size 2  which acts like a hinge for any generic 3-dimensional realisation.
It is currently an open problem as to what an exact combinatorial characterisation should be for higher dimensional (i.e., $d \geq 3$) realisations.

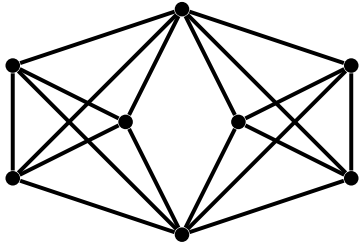
\begin{figure}[htp]
	\begin{center}
        \begin{tikzpicture}[scale=0.75]
			\node[vertex] (top) at (0,2) {};
			\node[vertex] (bottom) at (0,-2) {};

			\node[vertex] (11) at (-1,0) {};
			\node[vertex] (21) at (-3,1) {};
			\node[vertex] (31) at (-3,-1) {};
			
			\node[vertex] (12) at (1,0) {};
			\node[vertex] (22) at (3,1) {};
			\node[vertex] (32) at (3,-1) {};
			
			\draw[edge] (11)edge(top);
			\draw[edge] (21)edge(top);
			\draw[edge] (31)edge(top);
			
			\draw[edge] (top)edge(12);
			\draw[edge] (top)edge(22);
			\draw[edge] (top)edge(32);
   
			\draw[edge] (11)edge(bottom);
			\draw[edge] (21)edge(bottom);
			\draw[edge] (31)edge(bottom);
			
			\draw[edge] (bottom)edge(12);
			\draw[edge] (bottom)edge(22);
			\draw[edge] (bottom)edge(32);
			
			\draw[edge] (11)edge(21);
			\draw[edge] (21)edge(31);
			\draw[edge] (31)edge(11);
			
			\draw[edge] (12)edge(22);
			\draw[edge] (22)edge(32);
			\draw[edge] (32)edge(12);
		\end{tikzpicture}
	\end{center}
	\caption{The double-banana graph.}\label{fig:db}
\end{figure}

\subsection{Direction rigidity and parallel redrawings}\label{subsec:parallelredraw}

Given a pair of frameworks $(G,p)$ and $(G,q)$,
we say that $(G,q)$ is \emph{direction equivalent} to $(G,p)$ if for each edge $uv \in E$ there exists $\lambda_{uv} \in \mathbb{R}$ such that
\begin{align*}
    q(u)-q(v) = \lambda_{uv}(p(u)-p(v)).
\end{align*}
In some literature \cite{ww96,sm23},
any such direction-equivalent framework $(G,q)$ is said to be a \emph{parallel redrawing} of $(G,p)$.
A framework $(G,q)$ is \emph{homothetic} to $(G,p)$ if there exists a point $z$ and a scalar $\lambda \in \mathbb{R}$ such that $q(v) = \lambda p(v) +z$ for each $v \in V$.
If two frameworks are homothetic then they are also direction equivalent.
A framework $(G,p)$ is now said to be \emph{direction rigid} if either $|V|=1$, or $(G,p)$ has an affine dimension greater than 0 and every framework that is direction equivalent to $(G,p)$ is homothetic to $(G,p)$;
any framework that is not direction rigid is said to be \emph{direction flexible}.

Given a framework $(G,p)$ in $\mathbb{R}^d$,
we define $C(G,p)$ to be the set of all $d$-dimensional realisations $q$ where $(G,q)$ is direction equivalent to $(G,p)$,
and we define $T(G,p)$ to be the set of all $d$-dimensional realisations $q$ where $(G,q)$ is homothetic to $(G,p)$.
It is rather easy to see that $C(G,p)$ is linear space with $T(G,p)$ as a $(d+1)$-dimensional linear subspace.\footnote{If $(G,p)$ has an affine dimension of 0 then $\dim T(G,p) = d$, however this is a rather trivial case.}
Hence a framework $(G,p)$ is direction rigid if and only if $\dim C(G,p) = d+1$.

In contrast to the standard edge-length rigidity case,
direction rigidity has a known combinatorial characterisation for generic realisations in all dimensions.

\begin{theorem}[\cite{ww96}]
    Let $(G,p)$ be a framework in $\mathbb{R}^d$ with at least two vertices.
    If $(G,p)$ is direction-rigid then $G$ contains a spanning $(d,d+1)$-tight subgraph.
    Conversely,
    if $(G,p)$ is generic and $G$ contains a spanning $(d,d+1)$-tight subgraph,
    then $(G,p)$ is direction-rigid.
\end{theorem}

When we are dealing with frameworks in dimension 2, there is an interesting observation to be made.
For the following result, we define for any framework $(G,p)$ in $\mathbb{R}^2$ the congruent framework $(G,p^\perp)$ formed by rotating the framework $90^\circ$ clockwise around the origin.
The following two results are well-known and can also be found in \cite{ww96}.
For the sake of completeness, we include the brief proofs for both statements.

\begin{lemma}\label{l:paralleldrawing}
    Let $(G,p)$ be a framework in $\mathbb{R}^2$.
    Then $\ker R(G,p^\perp) = C(G,p)$.
\end{lemma}

\begin{proof}
    Choose any point $q \in (\mathbb{R}^d)^V$.
    Then
    \begin{align*}
        q \in C(G,p) \quad &\Leftrightarrow \quad (q(u) - q(v)) \cdot (p(u)^\perp - p(v)^\perp) = 0 \quad \text{ for all } uv \in E, \\
        &\Leftrightarrow \quad R(G,p^\perp) q = 0,
    \end{align*}
    and the desired result follows.
\end{proof}

\begin{theorem}\label{t:paralleldrawing}
    A framework $(G,p)$ in $\mathbb{R}^2$ is infinitesimally rigid if and only if it is direction rigid.
\end{theorem}

\begin{proof}
    If $(G,p)$ has affine dimension 0 then either $|V|=1$ and $(G,p)$ is infinitesimally rigid and direction rigid, or $|V| >1$ and it is both infinitesimally flexible and direction flexible.
    Note that $\rank R(G,p^\perp) = \rank R(G,p)$,
    since one matrix can be obtained from the other by multiplying some columns by $-1$ and then rearranging the columns.
    The result now follows from \Cref{l:paralleldrawing} and the observation that infinitesimal rigidity requires $\dim \ker R(G,p) = 3 = \dim T(G,p)$.
\end{proof}

\section{Tropical hypersurfaces, extremality and direction rigidity}\label{sec:hyper}

As the necessary background has been introduced in the previous sections we are now ready to present our approach to bridging the gap between tropical geometry and rigidity theory. We will first discuss how to construct a dual graph and subdivided Newton polytope for tropical hypersurfaces. This is followed by studying the rigidity properties of the dual objects. Explicitly, in \Cref{t:main} we link the notion of extremal tropical hypersurfaces to direction rigidity.

\subsection{Tropical hypersurfaces and reciprocal diagrams}

We recall that a tropical hypersurface $C \subset \mathbb{R}^d$ is a tropical variety of dimension $d-1$.
A useful property of tropical hypersurfaces lies in their partitioning of the space they exist in.
In particular,
the set $\mathbb{R}^d \setminus C$ is the disjoint union of finitely many open sets, and the closure of any of these connected components (which we shall denote by the set $V$) is a convex polytope (see \cite{Nisse+Sottile:16} for a discussion on higher dimensional convexity).
If the intersection of two elements $v, w \in V$ has dimension $d-1$,
then $v \cap w = \sigma \in \widetilde{E}$. 
Hence, for a tropical hypersurface, we can always form a finite simple graph $G$ with vertex set $V$ and edge set $E$ such that $vw \in E$ if and only if $v \cap w \in \widetilde{E}$.
Note that there exists a natural bijection between the set $\widetilde{E}$ of $(d-1)$-dimensional faces of $C$ and the edge set $E$ of the graph $G$.
We refer to the graph $G$ as the \emph{dual graph} of $C$.
Since the regions of $C$ cover $\mathbb{R}^d$,
it follows that the dual graph $G$ is connected.

Given a tropical hypersurface $C \subset \mathbb{R}^d$ with dual graph $G=(V,E)$,
choose a $(d-1)$-dimensional polytope $\sigma \in \widetilde{E}$ which is the intersection of two regions $v,w \in V$.
We now define $x_\sigma(v,w) \in \mathbb{Z}^d$ to be the smallest integer-valued vector perpendicular to $L_\sigma$,
oriented in the direction travelled when crossing $\sigma$ from face $w$ to face $v$.
We now describe a particular embedding of the dual graph that we will use throughout the paper.

\begin{definition}\label{def:dualframework}
    Given a tropical hypersurface $C \subset \mathbb{R}^d$ with dual graph $G=(V,E)$,
    a framework $(G,p)$ is a \emph{reciprocal diagram of $C$} if $p(v) -p(w)$ is a positive rational scaling of $x_\sigma(v,w)$ for every edge $vw \in E$.
\end{definition}

\subsection{Subdivided Newton polytopes}

Given the definition of a reciprocal diagram,
it is unclear whether such a framework should exist for any tropical hypersurface.
Because of this,
we now construct a type of reciprocal diagram every tropical hypersurface has.
With this in mind,
we now present the definition of the Newton polytope of a tropical polynomial and outline the construction of a specific subdivision called the dual subdivision of the Newton polytope. For further discussion of the following ideas, see \cite[Section 1.4]{Mikhalkin+Rau:19}. 

Given a tropical polynomial $f \in \mathbb{T}[z_1\, ,\dots , z_d]$, we define the following notions:
\begin{enumerate}
    \item the \emph{support} of $f$ is the set $\mathrm{supp}(f):= \{k = (k_1 \, \dots \, ,k_d) \in A \subset \mathbb{Z}^d \, : \, a_k \neq - \infty \}$,
    \item the \emph{Newton polytope} of $f$ is the convex set $\mathrm{NP}(f):= \mathrm{Conv}(\mathrm{supp}(f))$, where $\mathrm{Conv}$ denotes the convex hull of a set of points,
    \item the \emph{lifted support} of $f$ is the set of points $\mathcal{L}\mathrm{supp}(f):=\{(k,-a_k) \, : \, k \in \mathrm{supp}(f)\} \subset \mathbb{Z}^d \times \mathbb{R}$,
    \item the \emph{lifted Newton polytope} of $f$ is the convex set $\mathcal{L}\mathrm{NP}(f):=\mathrm{Conv}(\mathrm{supp}(f))$.
\end{enumerate}
The lifted Newton polytope $\mathcal{L}\mathrm{NP}(f)$ then induces a subdivision structure on the Newton polytope $\mathrm{NP}(f)$ obtained under a projection of the lower faces, $\mathbb{R}^n \times \mathbb{R} \rightarrow \mathbb{R}$. This subdivision is called the \emph{dual subdivision} of the Newton polytope of $f$, and denoted $\mathrm{SD}(f)$.

\begin{example}\label{Ex:dual-sub-div}
    See \Cref{Fig:dual-sub-div-exs} for examples of dual subdivisions related to tropical hypersurfaces that have been previously discussed. 
\end{example}
\begin{figure}
        \centering
        \begin{subfigure}[b]{0.3\textwidth}
            \centering
            \includegraphics{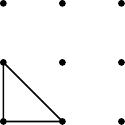}
        \caption{$\mathrm{SD}(x \oplus y \oplus 0)$, dual to the tropical hypersurfaces  in \Cref{Fig:exs-trop-var-line}.}
        \label{Fig:dual-sub-div-exs-(a)}
        \end{subfigure}
        \hfill
        \begin{subfigure}[b]{0.3\textwidth}
            \centering
            \includegraphics{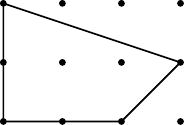}
        \caption{$\mathrm{SD}(0 \oplus x^{\otimes2} \oplus y^{\otimes2} \oplus x^{\otimes3} \otimes y)$, dual to the tropical hypersurfaces in \Cref{fig:Red-Trop-Var-Bad-W}.}
        \label{Fig:dual-sub-div-exs-(b)}
        \end{subfigure}
        \hfill
        \begin{subfigure}[b]{0.3\textwidth}
            \centering
            \includegraphics{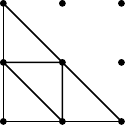}
        \caption{$\mathrm{SD}((-1) \oplus x \oplus y \oplus (x \otimes y) \oplus ((-1) \otimes y^{\otimes2}) \oplus ((-1) \otimes x^{\otimes 2}))$, dual to the tropical hypersurfaces  in \Cref{Fig:extremal-trop-quadratic}.}
        \label{Fig:dual-sub-div-exs-(c)}
        \end{subfigure}
        \caption{The three dual subdivisions from \Cref{Ex:dual-sub-div}.}
        \label{Fig:dual-sub-div-exs}
    \end{figure}

\begin{proposition}[{\cite[Proposition 3.1.6]{Mac+Stur:15}}]
    Let $f$ be a tropical polynomial.
    Then the 1-skeleton of the subdivided Newton polytope $\mathrm{SD}(f)$ is a reciprocal diagram of the tropical hypersurface $\mathbb{V}(f)$.
\end{proposition}

An explicit description of this duality is outlined in \cite[Theorem 2.3.7]{Mikhalkin+Rau:19}. 
The next result demonstrates how $\mathrm{SD}(f)$ encodes a specific balanced weighting for the tropical hypersurface $\mathbb{V}(f)$. 

\begin{proposition}[{\cite[Proposition 3.3.2]{Mac+Stur:15}}]\label{p:naturalweighting}
    Let $f$ be a tropical polynomial. The variety $\mathbb{V}(f)$ has a balanced weighting $\omega$ with weights equal to lattice lengths of the edges of $\mathrm{SD}(f)$. Explicitly, given a maximal face $\sigma$ of $\mathbb{V}(f)$, the weight $\omega(\sigma)$ is equal to the number of lattice points contained within the corresponding edge in $\mathrm{SD}(f)$ minus one. 
\end{proposition}

\subsection{Linking extremality and direction rigidity}

The following lemma describes the correspondence between reciprocal diagrams and rational balanced weightings. 

\begin{lemma}\label{lem:dualstress}
    Let $C$ be a tropical hypersurface in $\mathbb{R}^d$ with rational (possibly not balanced) weighting $\omega$. 
    For a given ridge $\tau \in \widetilde{V}$,
    let $(\sigma_1,\ldots,\sigma_n)$ be the unique cyclic ordering (up to orientation) of the elements of $\widetilde{E}$ containing $\tau$ such that $\sigma_i \cap \sigma_{i+1} = v_i$ for each $i \in \{1,\ldots,n-1\}$ and $\sigma_n \cap \sigma_1 = v_n$.
    Then, the equality
    \begin{align*}
        \sum_{i=1}^n \omega(\sigma_i) x_{\sigma_i}(v_{i+1},v_i) = 0
    \end{align*}
    holds for each $\tau \in \widetilde{V}$ (with $v_{n+1} = v_1$) if and only if $\omega$ is a rational balanced weighting of $C$.
\end{lemma}

\begin{proof}
 The rational weighting $\omega$ can be scaled by $\lambda \in \mathbb{Z}$, explicitly the least common multiple of the denominators of the weights in $\omega$ to obtain an integer weighting $\omega' = \lambda \omega$, without affecting the balancing. Then, as every weighted tropical hypersurface arises as the variety of a tropical polynomial, see \cite[Proposition 3.3.10]{Mac+Stur:15}, we have $C = \mathbb{V}(f)$ for $f \in \mathbb{T}[z_1, \, \dots \, ,z_d]$. 
    Furthermore,
    the weighting $\omega$ corresponds to the lattice length of the subdivided Newton polytope $SD(f)$, and the balancing condition at $\tau$ is equivalent to the edge vectors
    \begin{equation*}
        \omega'(\sigma_1) x_{\sigma_1}(v_{2},v_1) ~, \quad  \ldots \quad  , \quad   \omega'(\sigma_{n-1}) x_{\sigma_{n-1}}(v_{n},v_{n-1}) ~, \quad \omega'(\sigma_n) x_{\sigma_n}(v_{1},v_n)
    \end{equation*}
    of the two dimensional polygon face in $SD(f)$ associated to $\tau$ summing to zero. Therefore, for any chosen ridge $\tau$ we see that 
    \begin{align*}
        \sum_{i=1}^n \omega'(\sigma_i) x_{\sigma_i}(v_{i+1},v_i) = 0 &\iff  \sum_{i=1}^n \lambda\omega(\sigma_i) x_{\sigma_i}(v_{i+1},v_i) = 0 \\ 
        &\iff  \lambda \left( \sum_{i=1}^n \omega(\sigma_i) x_{\sigma_i}(v_{i+1},v_i) = 0\right),
    \end{align*}
    which demonstrates the desired result. 
\end{proof}

\begin{lemma}\label{lem:dualsums}
    Let $C$ be a tropical hypersurface in $\mathbb{R}^d$ with dual graph $G=(V,E)$ and rational balanced weighting $\omega$.
    Let $v_1,\ldots, v_n \in V$ and $\sigma_1,\ldots,\sigma_n$ be such that $v_i \cap v_{i+1} = \sigma_i$ for each $i \in \{1,\ldots,n-1\}$ and $v_n \cap v_1 = \sigma_n$.
    Then the sequence $(v_1,\ldots,v_n,v_1)$ is a cycle in $G$, and, given $v_{n+1} = v_1$, we have
    \begin{align*}
        \sum_{i=1}^n \omega(\sigma_i) x_{\sigma_i}(v_{i+1},v_i) = 0.
    \end{align*}
\end{lemma}

\begin{proof}
    The proof follows an analogous pattern to the previous lemma, by scaling the weighting and then utilising the subdivided Newton polytope of the corresponding polynomial. The equation now just describes an unscaled version of a closed path through $SD(f)$ in $\mathbb{R}^d$, which sums to zero.
\end{proof}

\begin{lemma}\label{lem:dualframework}
    Let $C$ be a tropical hypersurface in $\mathbb{R}^d$ with dual graph $G=(V,E)$.
    \begin{enumerate}
        \item \label{lem:dualframework1} Let $\omega$ be a rational balanced weighting of $C$.
        Then there exists a reciprocal diagram $(G,p)$ of $C$ such that for every $\sigma \in \widetilde{E}$ separating regions $v,w \in V$,
        we have
        \begin{align}\label{eq:dualframework}
            p(v) - p(w) = \omega(\sigma) x_\sigma(v,w).
        \end{align}
        Furthermore,
        $(G,p)$ is unique up to translation.
        \item \label{lem:dualframework2} Let $(G,p)$ be a reciprocal diagram of $C$.
        Then the map
        \begin{align*}
            \omega : \widetilde{E} \rightarrow \mathbb{Z}_{>0},~ \sigma \mapsto \frac{\|p(v) - p(w)\|}{\|x_\sigma(v,w)\|}
        \end{align*}
        is a rational balanced weighting.
    \end{enumerate}
\end{lemma}

\begin{proof}
    \ref{lem:dualframework1}:
    Fix a vertex $v_0 \in V$ and a spanning tree $T$ of $G$ (the existence of the latter stemming from the dual graph being connected).
    For each vertex $v \in V$,
    let $P_v = (v_0,\ldots,v_n)$ be the unique shortest path in $T$ from $v_0$ to $v = v_n$;
    note that $P_{v_i} = (v_0,\ldots,v_i)$ for each $i \in \{1,\ldots,n\}$.
    With this, we now construct our reciprocal diagram $(G,p)$ inductively as follows.
    First, we fix $p(v_0) = 0$.
    Next, fix a positive integer $n$ and assume that for every vertex $w \in V$ which is distance at most $n-1$ from $v_0$ we have already chosen $p(w)$.
    Now choose any $v \in V_n$ with minimal length path $P_v = (v_0,\ldots,v_n)$.
    Given $v_{n-1} \cap v = \sigma \in \widetilde{E}$,
    we now set $p(v) = p(v_{n-1}) + \omega(\sigma) x_\sigma(v,v_{n-1})$.

    We first observe that \cref{eq:dualframework} holds for every edge of $T$.
    Choose any edge $vw \in E$ that is not an edge of $T$ and fix $\sigma \in \widetilde{E}$ to be the maximal face that separates the regions $v,w \in V$.
    Since $T$ is a spanning tree,
    there exists a unique path $P = (v_1,\ldots,v_n)$ in $T$ with $v_1 = w$ and $v_n = v$.
    For each $i \in \{1,\ldots,n-1\}$,
    fix $\sigma_i \in \widetilde{E}$ to be the maximal face that separates the regions $v_i$ and $v_{i+1}$.
    Since $(v_1,\ldots,v_n,v_1)$ is a cycle,
    it follows from \Cref{lem:dualsums} that
    \begin{align*}
        p(v)- p(w) = \sum_{i=1}^{n-1} p(v_{i+1})- p(v_i) = \sum_{i=1}^{n-1} \omega(\sigma_i) x_{\sigma_i}(v_{i+1},v_i) = \omega(\sigma) x_\sigma(v,w).
    \end{align*}
    Hence \cref{eq:dualframework} holds for every edge of $G$.

    Finally,
    suppose there exists another reciprocal diagram $(G,q)$ that satisfies the same property as $(G,p)$.
    By translating $(G,q)$ we may assume that $q(v_0) = p(v_0)$.
    Now choose any vertex $v \in V$ with minimal length path $P_v = (v_0,\ldots,v_n)$.
    We now see that
    \begin{align*}
        q(v) =  q(v_0) + \sum_{i=1}^n q(v_i) - q(v_{i-1}) =  q(v_0) + \sum_{i=1}^n p(v_i) - p(v_{i-1}) = p(v) - p(v_0) + q(v_0) = p(v), 
    \end{align*}
    and so $q = p$.

    \ref{lem:dualframework2}:
    Choose any $\tau \in \widetilde{V}$.
    Fix the unique (up to orientation) cyclic ordering $(\sigma_1,\ldots,\sigma_n)$ of the elements of $\widetilde{E}$ that contain $\tau$ and the unique cyclic ordering $(v_1,\ldots,v_n)$ of the elements of $V$ that contain $\tau$ where $\sigma_i \cap \sigma_{i+1} = v_i$ for each $i \in \{1,\ldots,n-1\}$ and $\sigma_n \cap \sigma_1 = v_n$.
    Given $v_{n+1} = v_1$,
    we see that
    \begin{align*}
        \sum_{i=1}^n \omega(\sigma_i) x_{\sigma_i}(v_{i+1},v_i) = \sum_{i=1}^n p(v_{i+1}) - p(v_i) = 0.
    \end{align*}
    As this holds for each $\tau \in \widetilde{V}$,
    $\omega$ is a rational balanced weighting of $C$ by \Cref{lem:dualstress}.
\end{proof}

With this we are now ready to prove \Cref{t:main} and \Cref{cor:main}.

\begin{proof}[Proof of \Cref{t:main}]
    Fix $\omega$ to be the unique rational balanced weighting of $C$ that is associated to $(G,p)$ as given in \Cref{lem:dualframework}.
    By scaling both $(G,p)$ and $\omega$ by some positive rational scalar,
    we may suppose that $\omega$ is a balanced weighting (in that it takes only positive integer values).
    We may also suppose without loss of generality that $p(u) = 0$ for some fixed vertex $u \in V$.

    \ref{t:main1} $\implies$ \ref{t:main3}:
    Suppose that $(G,p)$ is not direction rigid.
    Then the linear space $C(G,p)$ has dimension at least $d+2$.
    Since $C(G,p)$ is defined solely by rational coefficient linear equations (i.e., the edge directions of $(G,p)$),
    there exists a rational point $q$ such that $q \neq p$ and $q(u)=p(u) = 0$.
    For sufficiently small and rational $\varepsilon > 0$,
    we note that the realisation $\tilde{p} := p +\varepsilon q$ is: (i) rational, (ii) has the property that $\tilde{p}(v) - \tilde{p}(w)$ is a positive scaling of $x_{\sigma}(v,w)$ for each $\sigma \in \widetilde{E}$ separating the regions $v$ and $w$, and (iii) is not a homothetic scaling of $p$ (since $p$ and $q$ are linearly independent).
    We now fix $p'$ to be the integer-valued realisation of $G$ formed from $\tilde{p}$ by some positive integer scaling.
    Fix $\omega'$ to be the balanced weighting of $C$ formed from $(G,p')$ as given in \Cref{lem:dualframework}.
    Since $p$ and $p'$ are linearly independent,
    it follows that $\omega$ and $\omega'$ are also linearly independent.
    Hence $C$ is not extremal.
    
    \ref{t:main3} $\implies$ \ref{t:main1}:
    Suppose that $C$ is not extremal.
    Then there exists a balanced weighting $\omega'$ of $C$ that is linearly independent of $\omega$.
    Fix $(G,p')$ to be the unique reciprocal diagram associated to $\omega'$ with $p'(u) = p(u) = 0$ as given in \Cref{lem:dualframework}.
    Since all reciprocal diagrams of $C$ are direction equivalent,
    both $(G,p)$ and $(G,p')$ are direction equivalent.
    Suppose for contradiction that $(G,p)$ and $(G,p')$ are homothetic.
    Then there exists some $\lambda \in \mathbb{R}$ such that $p'(v) = \lambda p(v)$ for all $v \in V$.
    It follows then from \Cref{lem:dualframework} that $\omega' = \lambda \omega$,
    contradicting that $\omega,\omega'$ are linearly independent.
    This now concludes the proof.    
\end{proof}

\begin{proof}[Proof of \Cref{cor:main}]
    The equivalence of \ref{cor:main1} and $(G,p)$ being direction rigid follows from \Cref{t:main}, and the equivalence of direction rigidity and infinitesimal rigid follows from \Cref{t:paralleldrawing}.
\end{proof}

\section{Tropical curves and planar rigidity}\label{sec:comb}

In this section, we apply our results from the previous section to the restricted class of \emph{tropical curves},
i.e., tropical hypersurfaces of $\mathbb{R}^2$.
Since a tropical curve only contains dimension 0 points (the elements of $\widetilde{V}$) and one-dimensional line segments and infinite rays (the elements of $\widetilde{E}$),
we now opt to refer to the elements of $\widetilde{V}$ as the \emph{vertices} of $C$ and the elements of $\widetilde{E}$ as the \emph{edges} of $C$.
We also refer to any infinite one-way ray as a \emph{half-edge} if we wish to differentiate them.

\subsection{Combining parallel redrawings with \texorpdfstring{\Cref{cor:main}}{the main corollary}}

An immediate consequence of \Cref{t:asimowroth} and \Cref{cor:main} is the following combinatorial corollary.

\begin{corollary}\label{cor:rigidgeneric}
    Let $C$ be an extremal tropical curve.
    Then the dual graph $G$ of $C$ is rigid in $\mathbb{R}^2$.
\end{corollary}

Sadly the converse of \Cref{cor:rigidgeneric} is not true.
To see this,
fix $C \subset \mathbb{R}^2$ to be the tropical curve of the tropical polynomial
\begin{equation}\label{eq:prismf}
    f(x,y) =0 \oplus (x^{\otimes 8}) \oplus (y^{\otimes 8}) \oplus (1 \otimes x \otimes y) \oplus (1 \otimes x^{\otimes 5} \otimes y) \oplus (1 \otimes x \otimes y^{\otimes 5}).
\end{equation}
The tropical curve that is constructed from $f$ is not extremal,
as can be seen in \Cref{fig:3prismflexibletropcurve}.

\begin{figure}[htp]
\begin{center}
    \includegraphics[width=0.75\textwidth]{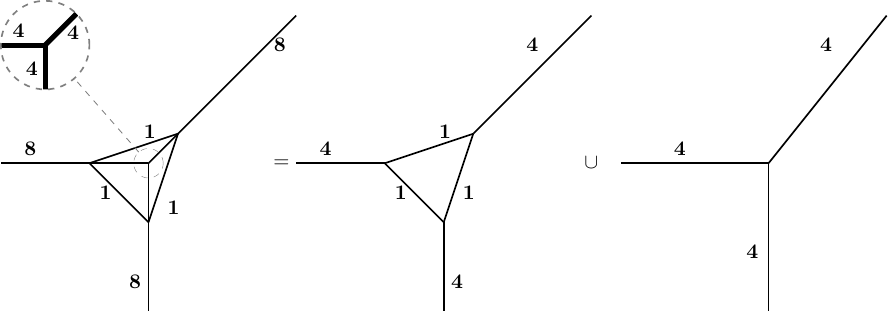}
\end{center}
\caption{(Left): The weighted tropical curve of the tropical polynomial given in \cref{eq:prismf}. (Right): Its decomposition into two weighted tropical curves.}
\label{fig:3prismflexibletropcurve}
\end{figure}

The 1-skeleton $(G,p)$ of the subdivided Newton polytope of $f$ is the framework pictured in \Cref{fig:3prismflexible}.
The rigidity matrix of $(G,p)$,
\begin{align*}
    R(G,p) = 
    \begin{bmatrix}
        -4 & 0 & 4 & 0 & 0 & 0 & 0 & 0 & 0 & 0 & 0 & 0 \\
        0 & -4 & 0 & 0 & 0 & 4 & 0 & 0 & 0 & 0 & 0 & 0 \\
        0 & 0 & 4 & -4 & -4 & 4 & 0 & 0 & 0 & 0 & 0 & 0 \\
        0 & 0 & 0 & 0 & 0 & 0 & -8 & 0 & 8 & 0 & 0 & 0 \\
        0 & 0 & 0 & 0 & 0 & 0 & 0 & -8 & 0 & 0 & 0 & 8 \\
        0 & 0 & 0 & 0 & 0 & 0 & 0 & 0 & 8 & -8 & -8 & 8 \\
        1 & 1 & 0 & 0 & 0 & 0 & -1 & -1 & 0 & 0 & 0 & 0 \\
        0 & 0 & -3 & 1 & 0 & 0 & 0 & 0 & 3 & -1 & 0 & 0 \\
        0 & 0 & 0 & 0 & -1 & 3 & 0 & 0 & 0 & 0 & 1 & -3
    \end{bmatrix}
\end{align*}
has rank $8 < 2 \cdot 6 - 3$,
hence $(G,p)$ is infinitesimally flexible and $C$ is not extremal by \Cref{t:main}.
However $G$ is $(2,3)$-tight and so rigid in $\mathbb{R}^2$ by \Cref{t:laman}.
Furthermore,
as $(G,p)$ is rigid\footnote{This is a consequence of the framework being prestress stable. See \cite{cw96} for more details on this concept.},
it also follows that infinitesimal rigidity cannot be replaced by rigidity in \Cref{cor:main}.

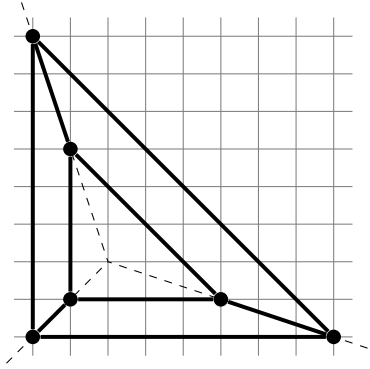
\begin{figure}[htp]
\begin{center}
    \begin{tikzpicture}[scale = 0.5]
	\draw[gray] (0,-0.5) -- (0, 8.5);
	\draw[gray] (1,-0.5) -- (1, 8.5);
	\draw[gray] (2,-0.5) -- (2, 8.5);
	\draw[gray] (3,-0.5) -- (3, 8.5);
	\draw[gray] (4,-0.5) -- (4, 8.5);
	\draw[gray] (5,-0.5) -- (5, 8.5);
	\draw[gray] (6,-0.5) -- (6, 8.5);
	\draw[gray] (7,-0.5) -- (7, 8.5);
	\draw[gray] (8,-0.5) -- (8, 8.5);
   
	\draw[gray] (-0.5,0) -- (8.5,0);
	\draw[gray] (-0.5,1) -- (8.5,1);
	\draw[gray] (-0.5,2) -- (8.5,2);
	\draw[gray] (-0.5,3) -- (8.5,3);
	\draw[gray] (-0.5,4) -- (8.5,4);
	\draw[gray] (-0.5,5) -- (8.5,5);
	\draw[gray] (-0.5,6) -- (8.5,6);
	\draw[gray] (-0.5,7) -- (8.5,7);
	\draw[gray] (-0.5,8) -- (8.5,8);
   
        \node[vertex] (11) at (1,1) {};
	\node[vertex] (12) at (5,1) {};
	\node[vertex] (13) at (1,5) {};
		
	\node[vertex] (21) at (0,0) {};
	\node[vertex] (22) at (8,0) {};
	\node[vertex] (23) at (0,8) {};

        \draw[dashed] (-0.3,8.9) -- (2,2);
        \draw[dashed] (8.9,-0.3) -- (2,2);
        \draw[dashed] (-0.7,-0.7) -- (2,2);
			
	\draw[edge] (11)edge(12);
	\draw[edge] (12)edge(13);
	\draw[edge] (13)edge(11);
		
	\draw[edge] (21)edge(22);
	\draw[edge] (22)edge(23);
	\draw[edge] (23)edge(21);
			
	\draw[edge] (11)edge(21);
	\draw[edge] (12)edge(22);
	\draw[edge] (13)edge(23);   
    \end{tikzpicture}
\end{center}
\caption{The 1-skeleton of the subdivided Newton polytope constructed from the tropical polynomial given in \cref{eq:prismf}, with the $\mathbb{Z}^2$ lattice represented by grey lines. The framework is not direction rigid as the vertices of the larger triangle can be slid up and down the dotted lines in unison whilst maintaining the correct edge directions.}
\label{fig:3prismflexible}
\end{figure}

\subsection{Applications of \texorpdfstring{\Cref{cor:rigidgeneric}}{corollary}}

We can use \Cref{cor:rigidgeneric} to determine structural properties for extremal tropical curves.
For the following results,
the degree of a vertex of a tropical curve is the number of edges it is contained in,
and a vertex is said to be \emph{trivalent} if it is contained in exactly three edges.

\begin{corollary}
    Let $C$ be an extremal tropical curve.
    Then $C$ contains a trivalent vertex.
    Furthermore,
    if $C$ contains exactly one trivalent vertex,
    then all other vertices of $C$ have degree 4 and $C$ contains exactly three half-edges.
\end{corollary}

\begin{proof}
    First suppose that $C$ that either: (i) $C$ has no trivalent vertices, (ii) $C$ has exactly one trivalent vertex but more than three half-edges, or (iii) $C$ has exactly one trivalent vertex, exactly three half-edges but a vertex of degree 5.
    If (i) or (ii) hold, then the dual graph $G=(V,E)$ of $C$ contains at most one triangle,
    while if (iii) holds then $G$ has exactly two triangles and a face with 5 sides.
    By Euler's formula for planar graphs,
    we see that in either case $G$ has at most $2|V|-4$ edges.
    Hence $G$ is flexible in $\mathbb{R}^2$ by \Cref{t:maxwell}.
    The result now follows from \Cref{cor:rigidgeneric}.
\end{proof}

\begin{corollary}
    Let $C$ be a tropical curve with at least 7 faces.
    If every face of $C$ has at most three sides,
    then $C$ is not extremal.
\end{corollary}

\begin{proof}
    If $C$ contains no face with 4 or more sides then its dual graph $G=(V,E)$ has maximal degree 3.
    By applying the hand-shaking lemma,
    we see that $|E| \leq 3|V|/2$.
    Since $|V| \geq 7$,
    it follows that $|E|<2|V|-3$ and so $G$ is flexible in $\mathbb{R}^2$ by \Cref{t:maxwell}.
    The result now follows from \Cref{cor:rigidgeneric}.
\end{proof}

\begin{corollary}
    Let $C$ be a tropical curve that contains two half-edges $\sigma,\sigma'$ that share a vertex $\tau$ but do not share a face.
    Then $C$ is not extremal.
\end{corollary}

\begin{proof}
    Suppose for contradiction that $C$ is extremal.
    By \Cref{cor:rigidgeneric},
    the dual graph $G=(V,E)$ of $C$ is rigid in $\mathbb{R}^2$.
    Let $e$ and $e'$ be the edges of $G$ that correspond to $\sigma$ and $\sigma'$ respectively.
    Since $\sigma,\sigma'$ do not share a face,
    $e$ and $e'$ are independent (i.e., they do not share a vertex).
    Label the two connected components of $\mathbb{R}^2 \setminus (\sigma \cup \sigma')$ as $A_1,A_2$.
    Note that any continuous path travelling from a face in $A_1$ to a connected component in $A_2$ must cross either $\sigma$ or $\sigma'$.
    Hence $\{e,e'\}$ is a separating edge set of $G$.
    Since $G$ is 2-connected (an immediate consequence of it being rigid in $\mathbb{R}^2$),
    $G - \{e,e'\}$ contains exactly two connected components.
    Let $X,Y \subset V$ be the vertex sets of the connected components of $G - \{e,e'\}$.
    As $e$ and $e'$ are independent,
    both $X$ and $Y$ contain at least two elements each.

    By \Cref{t:laman},
    there exists a $(2,3)$-tight spanning subgraph $H=(V,F)$ of $G$.
    The graph $H$ is 2-connected (since it is rigid in $\mathbb{R}^2$),
    and so $H$ contains both edges $e,e'$.
    We now note that
    \begin{align*}
        i_H(V) = i_H(X) + i_H(Y) + 2 \leq 2|X| - 3 + 2|Y| - 3 + 2 = 2|V| - 4,
    \end{align*}
    contradicting that $H$ is $(2,3)$-tight.
    This now concludes the proof.
\end{proof}

\section{Computing extremality in general tropical varieties}\label{sec:decomp}

In this section we describe computational methods for determining extremality and for constructing extremal decompositions.

\subsection{Rigidity matrices for tropical varieties}\label{subsec:rc}

Fix $C$ to be a $k$-dimensional tropical variety in $\mathbb{R}^d$.
Define $\widetilde{R}(C)$ to be the $|\widetilde{E}| \times d|\widetilde{V}|$ integer-valued matrix where for each $\sigma \in \widetilde{E}$ and each $\tau \in \widetilde{V}$ we set the $d$ coordinates corresponding to the pair $(\sigma,\tau)$ to be $z_{\tau}(\sigma)$ if $\tau \subset \sigma$ and $(0,\ldots,0)$ otherwise;
e.g., for each $\sigma \in \widetilde{E}$, the row of $\widetilde{R}(C)$ corresponding to $\sigma$ has the form
\begin{align*}
	\Big[ \quad \overbrace{0 ~ \cdots ~ 0 }^{\tau \notin \sigma} \qquad \cdots \qquad \overbrace{z_{\tau}(\sigma)^\top}^{\tau \in \sigma} \quad \Big].
\end{align*}
For each $\tau \in \widetilde{V}$,
fix $t=d-(k-1)$ linearly independent vectors $x_1(\tau), \ldots, x_t(\tau) \in L_\tau^\perp \cap \mathbb{Z}^d$,
and define the $d|\widetilde{V}| \times t |\widetilde{V}|$ integer-valued matrix $L(C)$ by setting the column corresponding to $(\tau,i) \in \widetilde{V}\times \{1,\ldots,t\}$ to be the vector with $x_i(\tau)$ for the $d$ rows corresponding to $\tau$ and $0$ elsewhere;
e.g., using the notation $\mathbf{0}$ for the $d$-dimensional zero vector, the $d$ rows of $L(C)$ corresponding to $\tau \in \widetilde{V}$ have the form
\begin{align*}
    \tau ~ \Big[ \quad \overbrace{x_1(\tau) ~ \cdots ~ x_t(\tau)}^{\tau}  \qquad \cdots \qquad \overbrace{\mathbf{0} ~ \cdots ~ \mathbf{0}}^{\tau' \neq \tau} \quad \Big].
\end{align*}
If $t= d$ (i.e., each element of $\widetilde{V}$ corresponds to a point) then we choose our $x_i(\tau)$ vectors to be the standard orthonormal basis of $\mathbb{R}^d$ so that $L(C)$ is simply an identity matrix.
We now define the $|\widetilde{E}| \times t|\widetilde{V}|$ integer-valued matrix $R(C) := \widetilde{R}(C)L(C)$.

\begin{lemma}\label{lem:rcweights}
    Let $C$ be a tropical variety with a weighting $\omega : \widetilde{E} \rightarrow \mathbb{Z}_{>0}$.
    Then the following properties are equivalent:
    \begin{enumerate}
        \item $\omega$ is a balanced weighting of $C$;
        \item $\omega^\top R(C) = 0$.
    \end{enumerate}
\end{lemma}

\begin{proof}
    We observe that $\omega^\top R(C) = 0$ if and only if the following equality holds for each $\tau \in \widetilde{V}$ and each $i \in \{1,\ldots,t\}$:
    \begin{align*}
        \sum_{\sigma \supset \tau} \omega(\sigma) z_{\tau}(\sigma) \cdot x_i(\tau)= 0.
    \end{align*}
    The above equation (when run over $i \in \{1,\ldots,t\}$) is equivalent to the balancing equation at $\tau$.
    This now concludes the proof.
\end{proof}

\begin{example}
    Take the tropical polynomial $f (x,y,z) = 0 \oplus x \oplus y \oplus z$,
    as depicted in \Cref{fig:tropical_hyperplane_3D}.
    The tropical hypersurface $\mathbb{V}(f)$ is a polyhedral cone with six maximal faces 
    \begin{align*}
        \sigma_{xy} = \{ (x,y,z) : z=0 \geq \max\{x,y\}\}, \qquad
        \sigma_{xz} = \{ (x,y,z) : y=0 \geq \max\{x,z\}\}, \\
        \sigma_{x0} = \{ (x,y,z) : y=z \geq \max\{x,0\}\}, \qquad
        \sigma_{yz} = \{ (x,y,z) : x=0 \geq \max\{y,z\}\}, \\
        \sigma_{y0} = \{ (x,y,z) : x=z \geq \max\{y,0\}\}, \qquad
        \sigma_{z0} = \{ (x,y,z) : x=y \geq \max\{z,0\}\},
    \end{align*}    
    and four ridges
    \begin{align*}
        \tau_{x} = \{ (x,y,z) : y=z=0 \geq x\}, \qquad
        \tau_{y} = \{ (x,y,z) : x=z=0 \geq y\}, \\
        \tau_{z} = \{ (x,y,z) : x=y=0 \geq z\}, \qquad
        \tau_{0} = \{ (x,y,z) : x=y=z \geq 0\}.
    \end{align*}
    We choose the $z_\tau(\sigma)$ vectors as follows:
    \begin{align*}
        \tau_{x} &: \qquad z_{\tau_{x}}(\sigma_{xy}) = (0,-1,0), \qquad z_{\tau_{x}}(\sigma_{xz}) = (0,0,-1), \qquad z_{\tau_{x}}(\sigma_{x0}) = (0,1,1);\\
        \tau_{y} &:  \qquad z_{\tau_{y}}(\sigma_{xy}) = (-1,0,0), \qquad z_{\tau_{y}}(\sigma_{yz}) = (0,0,-1), \qquad z_{\tau_{y}}(\sigma_{y0}) = (1,0,1);\\
        \tau_{z} &: \qquad z_{\tau_{z}}(\sigma_{xz}) = (-1,0,0), \qquad z_{\tau_{z}}(\sigma_{yz}) = (0,-1,0), \qquad z_{\tau_{z}}(\sigma_{z0}) = (1,1,0); \\
        \tau_{0} &: \qquad z_{\tau_{0}}(\sigma_{x0}) = (-1,0,0), \qquad z_{\tau_{z}}(\sigma_{y0}) = (0,-1,0), \qquad z_{\tau_{z}}(\sigma_{z0}) = (0,0,-1).
    \end{align*}
    With this, we see that
    \begin{align*}
        \widetilde{R}(\mathbb{V}(f)) = 
        \begin{bmatrix}
            0 & -1 & 0 & -1 & 0 & 0 & 0 & 0 & 0 & 0 & 0 & 0 \\
            0 & 0 & -1 & 0 & 0 & 0 & -1 & 0 & 0 & 0 & 0 & 0 \\
            0 & 1 & 1 & 0 & 0 & 0 & 0 & 0 & 0 & -1 & 0 & 0 \\
            0 & 0 & 0 & 0 & 0 & -1 & 0 & -1 & 0 & 0 & 0 & 0 \\
            0 & 0 & 0 & 1 & 0 & 1 & 0 & 0 & 0 & 0 & -1 & 0 \\
            0 & 0 & 0 & 0 & 0 & 0 & 1 & 1 & 0 & 0 & 0 & -1
        \end{bmatrix}.
    \end{align*}
    We also choose the $x_i(\tau)$ vectors as follows:
    \begin{align*}
        x_1(\tau_{y}) = x_1(\tau_{z}) = (1,0,0), \qquad  x_1(\tau_{x}) = x_2(\tau_{z}) = (0,1,0), \qquad x_2(\tau_{x}) = x_2(\tau_{y}) = (0,0,1),\\
        x_1(\tau_{0}) = (1,-1,0), \qquad x_2(\tau_{0}) = (1,1,-2).
    \end{align*}
    With this, we see that
    \begin{align*}
        L(\mathbb{V}(f)) = 
        \begin{bmatrix}
            0 & 0 & 0 & 0 & 0 & 0 & 0 & 0 \\
            1 & 0 & 0 & 0 & 0 & 0 & 0 & 0 \\
            0 & 1 & 0 & 0 & 0 & 0 & 0 & 0 \\
            0 & 0 & 1 & 0 & 0 & 0 & 0 & 0 \\
            0 & 0 & 0 & 0 & 0 & 0 & 0 & 0 \\
            0 & 0 & 0 & 1 & 0 & 0 & 0 & 0 \\
            0 & 0 & 0 & 0 & 1 & 0 & 0 & 0 \\
            0 & 0 & 0 & 0 & 0 & 1 & 0 & 0 \\
            0 & 0 & 0 & 0 & 0 & 0 & 0 & 0 \\
            0 & 0 & 0 & 0 & 0 & 0 & 1 & 1 \\
            0 & 0 & 0 & 0 & 0 & 0 & -1 & 1 \\
            0 & 0 & 0 & 0 & 0 & 0 & 0 & -2 
        \end{bmatrix}.
    \end{align*}
    Hence we obtain the rank 5 matrix
    \begin{align*}
        R(\mathbb{V}(f)) = \widetilde{R}(\mathbb{V}(f)) L(\mathbb{V}(f)) =
        \begin{bmatrix}
            -1 & 0 & -1 & 0 & 0 & 0 & 0 & 0 \\
            0 & -1 & 0 & 0 & -1 & 0 & 0 & 0 \\
            1 & 1 & 0 & 0 & 0 & 0 & -1 & -1 \\
            0 & 0 & 0 & -1 & 0 & -1 & 0 & 0 \\
            0 & 0 & 1 & 1 & 0 & 0 & 1 & -1 \\
            0 & 0 & 0 & 0 & 1 & 1 & 0 & 2
        \end{bmatrix}.
    \end{align*}
\end{example}

With this, we are now ready to prove \Cref{t:detect}.

\begin{proof}[Proof of \Cref{t:detect}]
    Fix $\omega$ to be a balanced weighting $\omega$ of $C$.
    By \Cref{lem:rcweights},
    $\omega \in \ker R(C)^\top$,
    and hence $\dim \ker R(C)^T \geq 1$.
    It follows from \Cref{lem:rcweights} that $n \leq \dim \ker R(C)^T$.
    Hence $\dim \ker R(C)^T = 1$ if and only if $n = 1$.
    
    Suppose that $\dim \ker R(C)^T >1$.
    Since $R(C)$ is an integer valued matrix,
    there exists elements $\omega_2,\ldots,\omega_n \in \ker R(C)^T \cap \mathbb{Z}^{\widetilde{E}}$ such that $\omega, \omega_2,\ldots,\omega_n$ are linearly independent.
    As every coordinate of $\omega$ is positive,
    for each $2\leq i \leq k$ there exists a sufficiently large scalar $a_i > 0$ such that
    $a_i \omega + \omega_i \in \mathbb{Z}_{> 0}^{\widetilde{E}}$.
    Define $\lambda_1 := \omega$ and $\lambda_i := a_i \omega + \lambda_i$ for each $2 \leq i \leq n$.
    By our construction,
    $\lambda_1,\ldots,\lambda_n$ are linearly independent vectors contained in the set $R(C) \cap \mathbb{Z}_{> 0}^{\widetilde{E}}$.
    The result now follows from \Cref{lem:rcweights}.
\end{proof}

\Cref{t:detect} informs us that, so long as we have obtained the vectors $z_\tau(\sigma)$ and $x_i(\tau)$ for each $\tau \in \widetilde{V}$,
we have a polynomial-time algorithm (with respect to $\widetilde{V}$) for determining whether a tropical variety is extremal.
We can also use it to obtain an inequality regarding the number of ridges and maximal faces of an extremal tropical variety.

\begin{corollary}\label{cor:extremalbound}
    Let $C$ be a $k$-dimensional extremal tropical variety in $\mathbb{R}^d$.
    Then
    \begin{align*}
        |\widetilde{E}| \leq (d-k+1)|\widetilde{V}| + 1.
    \end{align*}
\end{corollary}

\begin{proof}
    The number of columns ($(d-k+1)|\widetilde{V}|$) is an upper bound for the rank of $R(C)$.
    The result now follows from \Cref{t:detect}.
\end{proof}

The bound in \Cref{cor:extremalbound} is tight,
as can be observed from the extremal tropical curve given in \Cref{Fig:extremal-trop-quadratic} that has 4 vertices and $9 = 2 \cdot 4 + 1$ edges.
It is also best possible:
for example,
the extremal tropical curve given in \Cref{Fig:exs-trop-var}(b) has 9 vertices and $18 < 2 \cdot 9 + 1$ edges,
while the non-extremal tropical curve given in \Cref{fig:3prismflexibletropcurve} has 4 vertices and $9 = 2 \cdot 4 + 1$ edges.

\subsection{Decomposing tropical varieties into extremal components}

A \emph{convex cone} is a set $K \subseteq \mathbb{R}^n$ such that $ax+by \in K$ for all $x,y \in K$ and all $a,b \geq 0$.
A convex cone is also said to be \emph{strongly convex} if the equality $x + y = 0$ holds for two points $x,y \in K$ if and only if $x=y=0$.
A \emph{generating set} of a cone is a subset $A \subset K$ such that every element of $K$ is the sum of non-negative scalar copies of finitely many elements of $A$;
if a finite generating set exists then $K$ is said to be \emph{finitely generated}.
A \emph{minimal generating set} is any generating set of minimal cardinality over all possible generating sets.
Every finitely generated strongly convex cone has a minimal generating set, and the minimal generating set is unique up to positive scalar multiplication of its elements.

For a tropical variety $C$,
we fix $W(C)$ to be the set of rational partial balanced weightings of $C$.
With this, we are now ready to state our next result.

\begin{theorem}\label{t:allirreducible}
    Let $C$ be a tropical variety.
    If $C$ is not extremal then $W(C)$ is a finitely generated strongly convex cone with a minimal generating set $\Omega=\{\omega_1,\ldots,\omega_n\}$,
    with each $\omega_i$ being a strictly partial balanced weighting of $C$.
    Furthermore,
    given each $C_i$ is the tropical variety formed from the support of $\omega_i$,
    the set $\{C_1,\ldots,C_n\}$ is exactly the set of extremal tropical varieties contained in $C$.
\end{theorem}

\begin{proof}
    Since $W(C)$ is defined by finitely many rational hyperplanes,
    it is a finitely generated strongly convex cone with a minimal generating set contained in $\mathbb{Q}^{\widetilde{E}}_{\geq 0}$.
    We now scale the elements of our minimal generating set to obtain $\Omega \subset \mathbb{Z}^{\widetilde{E}}_{\geq 0}$.
    
    We now must prove no element of $\Omega$ is a balanced weighting.
    Suppose for contradiction that $\Omega \cap \mathbb{Z}^{\widetilde{E}}_{> 0} \neq \emptyset$.
    By relabelling $\Omega$ we may suppose that $\omega_1 \in \mathbb{Z}^{\widetilde{E}}_{> 0}$.
    By the minimality of $\Omega$ we have that $\omega_1,\omega_2$ are linearly independent.
    Hence there exists positive integers $a,b$ such that $\lambda := a \omega_1 - b \omega_2$ is a strictly partial balanced weighting.
    As $\omega_1$ is one of the minimal generators of $W(C)$ and $\omega_1 =  ( \lambda + b\omega_2)/a$,
    we must have that $\lambda = \sum_{i=1}^n a_i \omega_i$ for some non-negative scalars $a_1,\ldots,a_n$ where $a_1 >0$.
    However this is impossible,
    as $\omega_1 + x \in \mathbb{R}_{>0}^{\widetilde{E}}$ for all $x \in \mathbb{R}_{\geq 0}^{\widetilde{E}}$.
    Hence no element of $\Omega$ is a balanced weighting.
    
    It is now sufficient to show that the zero set of one generator does not contain the zero set of another.
    Suppose for contradiction that $\omega_i^{-1}(\{0\}) \subset \omega_j^{-1}(\{0\})$ for some $i\neq j$.
    Then there exists positive integers $a,b$ such that $\lambda := a \omega_1 - b \omega_2$ is a strictly partial balanced weighting and $\omega_i^{-1}(\{0\}) \subsetneq \lambda^{-1}(\{0\})$.
    As $\omega_i$ is one of the minimal generators of $W(C)$ and $\omega_i =  (b \omega_j + \lambda)/a$,
    we must have that $\lambda = \sum_{i=1}^n a_i \omega_i$ for some non-negative scalars $a_1,\ldots,a_n$ where $a_i >0$.
    However this is impossible,
    as $\omega_i^{-1}(\{0\}) \subsetneq \lambda^{-1}(\{0\})$.
    This now concludes the proof.
\end{proof}

Our aim now is to use \Cref{t:allirreducible} to break down tropical varieties into extremal parts.
To be more specific, we are aiming to obtain the following decomposition.

\begin{definition}
    A set of distinct extremal tropical varieties $\{C_1,\ldots,C_n\}$ contained in a tropical variety $C$ is said to be an \emph{extremal decomposition} if $C=\bigcup_{i=1}^n C_i$.
\end{definition}

The following lemma is an immediate consequence of the relationship between strict partial weightings proper extremal subvarieties.

\begin{lemma}\label{lem:deconstruct}
    Let $C$ be a tropical variety with strictly partial balanced weightings $\omega_1,\ldots,\omega_n$.
    Suppose that the linear span of $\omega_1,\ldots,\omega_n$ intersects $\mathbb{Q}^{\widetilde{E}}_{>0}$.
    Then,
    given each $C_i$ is the tropical variety formed from the support of $\omega_i$,
    we have $C = \bigcup_{i=1}^n C_i$.
\end{lemma}

\Cref{lem:deconstruct} can immediately provide us with an upper bound on the size of a minimal decomposition of a tropical curve.

\begin{proposition}\label{p:decomp}
    Let $C$ be a tropical variety.
    If $m = |\widetilde{E}| - \rank R(C)$,
    then there exists a decomposition of $C$ into $m$ pairwise-distinct extremal tropical varieties.
\end{proposition}

\begin{proof}
    Fix $\Omega=\{\omega_1,\ldots,\omega_n\}$ to be the minimal generating set of $W(C)$ given by \Cref{t:allirreducible}.
    The number of generators of a finitely generated convex cone is at least its dimension,
    hence $m \leq n$.
    Furthermore,
    the linear span of $\Omega$ must have the same dimension as $W(C)$,
    hence there exists a subset $\Omega' \subset \Omega$ where $|\Omega'|=m$ and the linear span of $\Omega'$ is equal to the linear span of $W(C)$.

    By relabelling elements of $\Omega$ we may suppose that $\Omega' = \{\omega_1,\ldots, \omega_m\}$.
    If $\bigcap_{i=1}^m \omega_i^{-1}(\{0\}) \neq \emptyset$ then the linear span of $W(C)$ is disjoint from $\mathbb{Q}^n_{>0}$,
    which contradicts that $C$ is a tropical variety (and so has a balanced weighting).
    Hence the linear span of $\Omega'$ intersects with $\mathbb{Q}^{\widetilde{E}}_{>0}$.
    The result now follows from \Cref{lem:deconstruct}.
\end{proof}

If $C$ is a tropical variety with $\rank R(C) = |\widetilde{E}| - 2$,
then it follows from \Cref{p:decomp} that there exists a decomposition of $C$ into 2 distinct extremal tropical varieties.
As extremal tropical varieties cannot contain other extremal tropical varieties as proper subsets,
it clear in this case that this is the least number of extremal tropical varieties that $C$ can be decomposed into.
If, however, $\rank R(C) < |\widetilde{E}| - 2$,
then it is possible for there to exist a decomposition of $C$ into 2 distinct extremal tropical varieties.
This is illustrated in the following example.

\begin{example}
    Let $C$ is the tropical curve described in \Cref{Fig:Reducible-Trop-Var}.
    Given that the matrix associated to $C$ is of the form
    \begin{align*}
        R(C) = 
        \begin{bmatrix}
            1 & 0 \\
            -1& 0 \\
            0 & 1 \\
            0 &-1 \\
            1 & 1 \\
            -1&-1
        \end{bmatrix},
    \end{align*}
    and hence $\rank R(C) = 2 = |\widetilde{E}| - 4$.
 It follows from \Cref{p:decomp} that there exists a decomposition of $C$ into four pair-wise distinct extremal tropical varieties. However, $C$ contains exactly five extremal tropical varieties, and it also has a decomposition into two extremal tropical varieties. 
\end{example}

\begin{example}\label{Ex:triple-union-trop-var}
    We now describe a tropical curve that only decomposes into three extremal tropical curves.
    Take the tropical polynomial
    \begin{align*}
        f(x,y) &= \Big(x \oplus \big( (-1) \otimes y \big) \oplus 0 \Big) \otimes \Big( \big( (-1) \otimes x \big) \oplus y \oplus 0 \Big) \otimes \Big( \big( (-2) \otimes x \big) \oplus \big( (-2) \otimes y \big) \oplus 0 \Big) \\
        &= \big( (-3) \otimes x^{\otimes 3} \big) \oplus \big( (-2) \otimes x^{\otimes 2} \otimes y \big) \oplus \big( (-2) \otimes x \otimes y^{\otimes 2} \big) \oplus \big( (-3) \otimes y^{\otimes 3} \big)\\
        &\quad  \oplus \big( (-1) \otimes x^{\otimes 2} \big)\oplus \big( x \otimes y \big) \oplus x \oplus \big( (-1) \otimes y^{\otimes 2} \big) \oplus y \oplus 0.
    \end{align*}
    The tropical hypersurface $C$ that comes from $f$ can be seen in \Cref{fig:Triple-Union-trop-var}.
    The matrix associated to $C$ is of the form
    \begin{align*}
        R(C) = 
        \begin{bmatrix}
            -1& 0 & 0 & 0 & 0 & 0 & 0 & 0 & 1 & 0 & 0 & 0 \\
            0 & -1& 0 & 0 & 0 & 0 & 0 & 1 & 0 & 0 & 0 & 0 \\
            1 & 1 & 0 & 0 & 0 & 0 & 0 & 0 & 0 & 0 & 0 & 0 \\
            0 & 0 &-1 & 0 & 0 & 0 & 0 & 0 & 0 & 0 & 1 & 0 \\
            0 & 0 & 0 &-1 & 0 & 0 & 0 & 0 & 0 & 0 & 0 & 0 \\
            0 & 0 & 1 & 1 & 0 & 0 &-1 &-1 & 0 & 0 & 0 & 0 \\
            0 & 0 & 0 & 0 &-1 & 0 & 0 & 0 & 0 & 0 & 0 & 0 \\
            0 & 0 & 0 & 0 & 0 &-1 & 0 & 0 & 0 & 0 & 0 & 1 \\
            0 & 0 & 0 & 0 & 1 & 1 & 0 & 0 &-1 &-1 & 0 & 0 \\
            0 & 0 & 0 & 0 & 0 & 0 & 1 & 1 & 0 & 0 & 0 & 0 \\
            0 & 0 & 0 & 0 & 0 & 0 & 0 &-1 & 0 & 0 & 0 & 0 \\
            0 & 0 & 0 & 0 & 0 & 0 & 0 & 0 & 1 & 1 & 0 & 0 \\
            0 & 0 & 0 & 0 & 0 & 0 & 0 & 0 &-1 & 0 & 0 & 0 \\
            0 & 0 & 0 & 0 & 0 & 0 & 0 & 0 & 0 & 0 &-1 & 0 \\
            0 & 0 & 0 & 0 & 0 & 0 & 0 & 0 & 0 & 0 & 0 &-1
        \end{bmatrix},
    \end{align*}
    and hence $\rank R(C) = 12 = |\widetilde{E}| - 3$.
    A basis for the left kernel of $R(C)$ is given by the three balanced partial weightings (here represented as left kernel row vectors)
    \begin{align*}
        [ ~ 1 ~ 1 ~ 1 ~ 0 ~ 0 ~ 0 ~ 0 ~ 0 ~ 0 ~ 0 ~ 1 ~ 0 ~ 1 ~ 0 ~ 0 ~ ], \\
        [ ~ 0 ~ 0 ~ 0 ~ 1 ~ 1 ~ 1 ~ 0 ~ 0 ~ 0 ~ 1 ~ 0 ~ 0 ~ 0 ~ 1 ~ 0 ~ ], \\
        [ ~ 0 ~ 0 ~ 0 ~ 0 ~ 0 ~ 0 ~ 1 ~ 1 ~ 1 ~ 0 ~ 0 ~ 1 ~ 0 ~ 0 ~ 1 ~ ].
    \end{align*}
    From this, we observe that these vectors form a minimal generator set of $W(C)$,
    and thus the supports of these balanced partial weightings are the only extremal tropical varieties contained in $C$.
    Hence $C$ has a unique decomposition into 3 extremal tropical curves, as shown in \Cref{SubFig:triple-union-extr-decomp}.
\end{example}

\begin{figure}
    \centering
    \begin{subfigure}[b]{0.45\textwidth}
            \centering
            \includegraphics{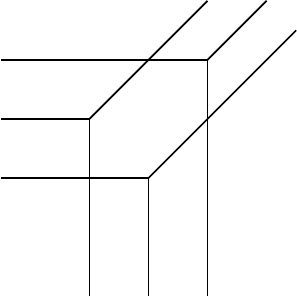}
        \caption{}
        \label{SubFig:triple-union-full-var}
        \end{subfigure}
        \begin{subfigure}[b]{0.45\textwidth}
            \centering
            \includegraphics{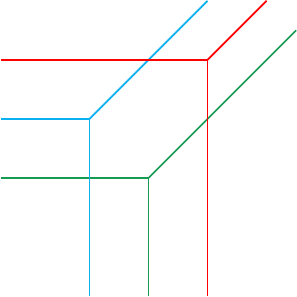}
        \caption{}
        \label{SubFig:triple-union-extr-decomp}
        \end{subfigure}
    \caption{The tropical variety of the tropical polynomial from \Cref{Ex:triple-union-trop-var} in \ref{SubFig:triple-union-full-var}, with colours indicating the extremal decomposition in \ref{SubFig:triple-union-extr-decomp}. }
    \label{fig:Triple-Union-trop-var}
\end{figure}

For any tropical variety $C$,
we define the convex polytope
\begin{align*}
    P(C) := W(C) \cap \left\{ x \in \mathbb{R}^{\widetilde{E}} : \sum_{\sigma \in \widetilde{E}} x_\sigma = 1 \right\}.
\end{align*}
The vertices of the polytope $P(C)$ form a minimal generating set of $W(C)$.
In fact, the vertices of $P(C)$ gift us a concrete method for constructing extremal decompositions of $C$.

\begin{theorem}\label{t:decompperfect}
    Let $C$ be a $k$-dimensional tropical variety in $\mathbb{R}^d$,
    let $\omega_1,\ldots,\omega_n$ be the vertices of the polytope $P(C)$,
    and let $C_i$ be the the tropical variety of $C$ formed from the support of $\omega_i$ for each $i \in \{1,\ldots,n\}$.
    Then $C_1, \ldots, C_n$ are the distinct extremal tropical varieties contained in $C$.
    Furthermore,
    for each subset $S \subset \{1,\ldots,n\}$,
    the following two statements are equivalent.
    \begin{enumerate}
        \item \label{t:decompperfect1} $\{C_i : i \in S\}$ is an extremal decomposition of $S$.
        \item \label{t:decompperfect2} The vertices $\omega_1,\ldots,\omega_n$ span $P(C)$;
        i.e., there exists scalars $\{t_i \in (0,1] : i \in S\}$ that satisfy $\sum_{i \in S} t_i = 1$ such that $\sum_{i \in S} t_i \omega_i$ is contained in the relative interior of $P(C)$.
    \end{enumerate}
\end{theorem}

\begin{proof}
    By \Cref{t:allirreducible},
    every extremal tropical variety contained in $C$ corresponds to a unique vertex of $P(C)$.
    Hence $C_1, \ldots, C_n$ are the distinct extremal tropical varieties contained in $C$.
    If \ref{t:decompperfect1} holds then $\omega = \frac{1}{|S|}\sum_{i\in S} \omega_i$ has non-zero coordinates.
    Thus $\omega$ lies in the relative interior of $P(C)$ and \ref{t:decompperfect2} holds.
    \ref{t:decompperfect2} implies \ref{t:decompperfect1} now follows from \Cref{lem:deconstruct}.
\end{proof}

\Cref{t:decompperfect} allows us to determine whether or not a tropical variety has a unique extremal decomposition.

\begin{corollary}\label{c:decompperfect}
    Let $C$ be a $k$-dimensional tropical variety in $\mathbb{R}^d$.
    Then the following three statements are equivalent.
    \begin{enumerate}
        \item \label{c:decompperfect1} $W(C)$ has a minimal generating set of size $n$ and $\rank R(C) = |\widetilde{E}| - n$.
        \item \label{c:decompperfect2} The convex polytope $P(C)$ is a $(n-1)$-dimensional simplex.
        \item \label{c:decompperfect3} $C$ has a unique extremal decomposition, and the decomposition contains $n$ extremal tropical varieties.
    \end{enumerate}
\end{corollary}

\begin{proof}
    Since the cone $W(C)$ has dimension $|\widetilde{E}| - \rank R(C)$,
    it follows that \ref{c:decompperfect1} and \ref{c:decompperfect2} are equivalent.
    Fix $n = \dim P(C) + 1$,
    and let $\omega_1,\ldots, \omega_m$ be the vertices of $P(C)$.
    Observe that the following three statements are equivalent: 
    (i) $P(C)$ is a simplex, (ii) $P(C)$ has $n$ vertices (i.e., $m=n$), and
    (iii) if a convex combination $\sum_{i=1}^m t_i \omega_i$ is contained in the relative interior of $P(C)$, then $t_i >0$ for each $i \in \{1,\ldots,m\}$.
    With this observation,
    the equivalence between \ref{c:decompperfect2} and \ref{c:decompperfect3} follows from \Cref{t:decompperfect}.
\end{proof}

Interestingly,
if a tropical variety does not have a unique extremal decomposition,
then an improved upper bound is possible for the number extremal tropical varieties needed to cover it.

\begin{corollary}\label{c:decompimperfect}
    Let $C$ be a $k$-dimensional tropical variety in $\mathbb{R}^d$ and fix $n = |\widetilde{E}| - \rank R(C)$.
    If $C$ has at least two distinct extremal decompositions,
    then there exists a decomposition of $C$ into at most $\lfloor \frac{n-1}{2} \rfloor + 1$ extremal tropical varieties.
\end{corollary}

\begin{proof}
    As $C$ has at least two distinct extremal decompositions,
    it follows from \Cref{c:decompperfect} that the polytope $P(C)$ is not a $(n-1)$-dimensional simplex.
    By \cite[pg.~123]{grunbaum},
    any $D$-dimensional polytope that is not a simplex must contain a set of at most $\lfloor \frac{D}{2} \rfloor + 1$ vertices that are not contained in a face.
    Hence,
    $P(C)$ contains a set of at most vertices $m \leq  \lfloor \frac{n-1}{2} \rfloor + 1$ that span $P(C)$.
    The result now follows from \Cref{t:decompperfect}.
\end{proof}

\subsection{An algorithm for decomposing tropical varieties}\label{subsec:alg}

As mentioned in the previous subsection,
there is a one-to-one correspondence between extremal tropical varieties contained in a tropical variety $C$ and the vertices of the convex polytope $P(C)$.
\Cref{c:decompperfect} further informs us that the problem of finding an extremal decomposition is equivalent to finding a subset of vertices that span $P(C)$.
The question now is: How can we efficiently construct extremal decompositions of a tropical variety?
Before we even begin, there is the question of what we even mean by a tropical variety.
For this, we will settle on the more geometric interpretation.
To be specific: if we refer to the input of an algorithm being a tropical variety, we are referring to a labelled set of maximal faces and ridges (each described by finitely many linear equalities and inequalities) where it is known which maximal faces intersect at what ridges.

An efficient method for finding an extremal decomposition of our chosen $k$-dimensional tropical variety $C \subset \mathbb{R}^d$ goes as follows.
\begin{enumerate}
    \item Generate the matrix $R(C)$. Forming $R(C)$ solely consists of finding the various $x_i(\tau)$ and $z_\tau(\omega)$ vectors. 
    Each of these vectors can be found using a combination of algorithms that put integer-valued matrices into Hermite normal form and Gaussian elimination.
    Hence, for a given ridge $\tau$ contained in a maximal face $\sigma$,
    each of the vectors $x_i(\tau)$ and $z_\tau(\omega)$ can be constructed in polynomial time with respect to $d,k$.
    As we must use the above two computational algorithms for every maximal face/ridge, constructing $R(C)$ will be completed in a polynomial number of steps with respect to $d$, $k$, $|\widetilde{V}|$ and $|\widetilde{E}|$.
    
    \item Compute the rank of $R(C)$. If $\rank R(C) = |\widetilde{E}| - 1$ then we terminate the algorithm here as $C$ is extremal.
    
    \item Find a vertex $\omega$ of $P(C)$.
    This can be performed using a variety of fast deterministic algorithms, including the simplex method and the ellipsoid method.
    While the ellipsoid method is guaranteed to run in polynomial time with respect to the size of $R(C)$ and the bit-size of the entries in $R(C)$,
    the simplex method is often significantly faster in practice. See \cite{bt97} for further discussion about these two algorithms.

    \item Set $S = \{ \omega\}$, fix the subset $I \subsetneq \widetilde{E}$ of zero coordinates of $x$, and set $I' = I$. For each $i \in I$, fix $b_i$ to be vector in $\mathbb{R}^{\widetilde{E}}$ with $b_i(i) = 1$ and $b_i(j) = 0$ if $j \neq i$. As $\omega$ is a vertex of $P(C)$,
    it is the unique element of $P(C)$ that satisfies $\omega^\top b_i = 0$ for each $i \in I$.

    \item \label{repeat1} Choose a set $J \subset I$ of size $\dim \ker R(C)^\top - 2$ where $I' \setminus J \neq \emptyset$. From this, define the matrix
    \begin{equation*}
        M_J = \big[\, R(C) ~ \ldots ~ \overbrace{b_j}^{j \in J} ~ \ldots \, \big].
    \end{equation*}
    By construction, the transpose of $M_J$ has a nullity of 1.
    Furthermore, there exists a unique non-zero rational element $x_J$ contained in $\ker M_J^\top$ where $x_J(i) \geq 0$ for every $i \in I$ and $x_J(j) > 0$ for some $j \in I' \setminus J$.

    \item \label{repeat2} Choose $t>0$ to be the largest value so that $\omega + t x_J \in P(C)$.
    From this, append $\omega + t x_J$ to $S$ and remove any elements from $I'$ which correspond to a non-zero coordinate of $\omega + t x_J$.

    \item Repeat steps \ref{repeat1} and \ref{repeat2} until the set $I'$ is empty.
    The set $S$ is now a set of vertices that span $P(C)$.
    An extremal decomposition of $C$ can now be constructed from $S$ via \Cref{t:decompperfect}.
\end{enumerate}

\subsection*{Acknowledgement}
SD and JM were supported by the Heilbronn Institute for Mathematical Research.

\bibliographystyle{alpha} 
\bibliography{references}

\end{document}